\documentclass[12pt,reqno]{amsart}
\usepackage{amssymb}
\usepackage{mathrsfs,mathtools}
\usepackage{palatino}
\usepackage{bbm}
\usepackage{pictex}
\usepackage{amscd}
\usepackage{tikz-cd}
\usepackage{latexsym}
\usepackage{amssymb}
\usepackage{enumitem}
\usepackage{eucal}
\usepackage{MnSymbol}
\usepackage{verbatim}
\usepackage{bbm}
\usepackage{hyperref}
\hypersetup{colorlinks,linkcolor=blue,urlcolor=cyan,citecolor=red}

\newcommand{\fg}{\mathfrak{g}}

\newcommand{\kk}{\mathbbm{k}}

\newcommand{\Bin}{\operatorname{Bin}}

\newcommand{\HC}{\operatorname{HC}}

\allowdisplaybreaks[2]

\numberwithin{equation}{section}
\newtheorem{theorem}{Theorem}[section]
\newtheorem{proposition}[theorem]{Proposition}
\newtheorem{lemma}[theorem]{Lemma}
\newtheorem{corollary}[theorem]{Corollary}
\newtheorem*{theorem*}{Theorem}
\newtheorem*{lemma*}{Lemma}
\newtheorem*{corollary*}{Corollary}
\theoremstyle{remark}
\newtheorem*{remark*}{Remark}

\title[Integral forms of Yangians]{Integral forms for Yangians of classical type}
\author[Hao Chang, Hongmei Hu, \lowercase{and}  Yue Hu  ]{Hao Chang, Hongmei Hu, \lowercase{and} Yue Hu  }
\address[H. Chang]{School of Mathematics and Statistics, Central China
Normal University, Wuhan 430079, China}
\email{chang@ccnu.edu.cn}
\address[H. Hu]{Department of Mathematics, Shanghai Maritime University, Shanghai 201306, China}
\email{hmhu@shmtu.edu.cn}
\address[Y. Hu]{School of Mathematical Sciences, Beijing University of Posts and Telecommunications, Beijing 102206, China}
\email{huyue@bupt.edu.cn}


\begin{document}

\begin{abstract}
Let $\mathfrak g$ be a simple Lie algebra of classical type.
We construct a $\mathbb Z[1/2]$-integral form of the Yangian
$Y(\mathfrak g)$ generated by divided powers of the positive
and negative Drinfeld generators, and establish its triangular
decomposition and PBW basis.
Under the loop filtration, its associated graded algebra is
the current algebra part of Garland's integral form. As applications, over an algebraically closed field of
characteristic $p>3$, we identify the small Yangian
with the restricted Yangian and classify their
finite dimensional simple modules in terms of restricted
Drinfeld polynomials.
\end{abstract}

\maketitle

\section{Introduction}
\label{sec:introduction}

Integral forms generated by divided powers are used to study
Lie algebras over fields of positive characteristic and
quantum groups at roots of unity. For semisimple Lie algebras, Kostant's integral form \cite{Kos} provides stable lattices in finite dimensional highest weight modules in characteristic zero, whose reductions yield highest weight modules in positive characteristic \cite{JanG}.
Garland \cite{Gar} extended the construction to untwisted
affine Lie algebras. Using Garland's integral form, Jakelić and Moura \cite{JM} showed that finite dimensional simple modules of hyper loop algebras can be obtained as quotients of reductions of suitable characteristic zero highest weight modules.  For quantized enveloping algebras, Lusztig \cite{Lus}
developed integral forms using quantum divided powers and
studied their specializations at roots of unity.  Chari and Pressley \cite{CP}  extended Garland's approach to quantum affine algebras, establishing the corresponding Cartan elements and triangular decomposition. 
They classified the finite dimensional simple type~I
modules at roots of unity by Drinfeld polynomials. In this paper, we study the corresponding questions for Yangians, which are filtered deformations of enveloping algebras of current Lie algebras. We construct a \(\mathbb Z[1/2]\)-integral form for Yangians of all classical types, identify the small Yangians with restricted Yangians, and classify their finite dimensional simple modules in characteristic \(p>3\).

Let $\mathfrak g$ be a simple Lie algebra of
type $A$, $B$, $C$, or $D$, and put
$R=\mathbb Z[1/2]$.
We use Drinfeld's presentation of the Yangian
$Y(\mathfrak g)$ \cite{D} (see also \cite{Mol07} and \cite{JLM}),
with generators
$\{e_{i,r},f_{i,r},h_{i,r}\mid i\in I,\ r\geq0\}$,
where $I$ indexes the simple roots.
Define
\[
\mathcal Y_R(\mathfrak g)
=
R\left\langle
e_{i,r}^{(m)},f_{i,r}^{(m)}
\,\middle|\,
i\in I,\ r,m\geq0
\right\rangle
\subseteq Y(\mathfrak g),
\qquad
x^{(m)}=\frac{x^m}{m!}.
\]
For quantum groups of finite or affine type, Lusztig's braid group actions provide a standard way to construct 
PBW root vectors.
This approach does not directly apply to the
divided power integral form of the Yangian considered here.
In \cite{CHH},  a shuffle algebra realization was used
to construct an $R$-basis of the positive part
consisting of ordered products of divided powers
of Lyndon root vectors.
The corresponding result for the negative part
follows by an anti-automorphism. To describe the full integral form, one must identify the
Cartan elements arising from products of positive
and negative divided powers and establish reordering
formulas over $R$. For applications to highest weight representations, these Cartan elements should also encode the
Drinfeld polynomials through their action on
highest vectors.

The Cartan elements we construct are the coefficients
of the series
\[
Q_i(u)=1+\sum_{m\geq1}P_{i,m}u^{-m},
\]
defined by
\[
(1+d_i/u)^{h_{i,0}}Q_i(u+d_i)
=
H_i(u)Q_i(u),
\]
where
$H_i(u)=1+d_i\sum_{r\geq0}h_{i,r}u^{-r-1}$
and $d_i$ are the symmetrizing factors of Cartan matrix.
This normalization is chosen so that, on a highest vector $v$ with highest series
$D_i(u+d_i)/D_i(u)$, one has
\[
Q_i(u)v=u^{-\deg D_i}D_i(u)v.
\]
Thus the coefficients $P_{i,m}$ recover those of
the Drinfeld polynomials.
Chari and Pressley \cite{CP} used a similar
normalization for the Cartan series of quantum
affine algebras. In the Yangian case, however, the shift
$u\mapsto u+d_i$ and the factor
$(1+d_i/u)^{h_{i,0}}$ give rise to lower
loop degree terms involving $h_{i,0}$
in the coefficients of $Q_i(u)$. Consequently, the Yangian analogue of
Chari--Pressley's basic lemma \cite[Lemma 5.1]{CP} requires
a further correction to the Cartan series.

To explain this difference, consider the rank one case
$Y(\mathfrak{sl}_2)$, and write
$e_r,f_r,h_r$ for its Drinfeld generators and
$Q(u)=1+\sum_{m\geq1}P_m u^{-m}$
for the corresponding Cartan series.
In Chari--Pressley's basic lemma for
$U_q(\widehat{\mathfrak{sl}}_2)$, the Cartan
projection of $(x_0^+)^{(n)}(x_1^-)^{(n)}$
is a multiple of the $n$-th coefficient of
their Cartan series, as in Garland's classical
identity (\cite[Lemma 7.5]{Gar}).
For the Yangian, our rank one
basic lemma (Theorem~\ref{thm:rank-basic}) gives
\[
\operatorname{HC}\bigl(e_0^{(n)}f_1^{(n)}\bigr)
=
(-1)^n[u^{-n}]
\frac{Q(u)}{B_{h_0-n+1}(u)},
\]
where $\operatorname{HC}$ denotes the Cartan projection.
The additional factor $B_{h_0-n+1}(u)^{-1}$ is needed because the normalization
of $Q(u)$ by Drinfeld polynomials does not
directly give the Cartan terms arising from
these divided power products.
The coefficients of $B_x(u)$ and its inverse
are integer valued polynomials in $x$,
so these corrections preserve integrality.
Moreover, the resulting identity is triangular
in the coefficients $P_m$ and allows us to prove
that they belong to the subalgebra generated
by divided powers (see \eqref{eq:QPtriangular}).

To extend the rank one identity to arbitrary products
$e_r^{(a)}f_s^{(b)}$, we use formal automorphisms
$\Phi_{c,z}$ of $Y(\mathfrak{sl}_2)[[z]]$
(cf. \cite{Zhang}).
We prove that they preserve the positive and negative
divided power integral forms
(Proposition~\ref{prop:prelim-formal-automorphisms})
and use them to establish the integrality of the
coefficients of divided powers of the generating series.
Together with the Cartan commutation relations and
the bases of the positive and negative parts from
\cite{CHH}, these results yield a triangular
decomposition of the full integral form in all
classical types.
More precisely, let $\mathcal H_R(\mathfrak g)$
be the $R$-algebra generated by
$\binom{h_{i,0}}{b}$ and $P_{i,m}$.
We prove that multiplication induces an isomorphism
of $R$-modules
\[
 \mathcal Y_R^{<}(\mathfrak g)
 \otimes_R\mathcal H_R(\mathfrak g)
 \otimes_R\mathcal Y_R^{>}(\mathfrak g)
 \xrightarrow{\;\sim\;}
 \mathcal Y_R(\mathfrak g).
\]
This gives an explicit $R$-basis of
$\mathcal Y_R(\mathfrak g)$
(Theorem~\ref{thm:general-pbw}).
Moreover, under the loop filtration, the leading
terms of $P_{i,m}$ are Garland's Cartan generators,
and
\[
 \operatorname{gr}\mathcal Y_R(\mathfrak g)
 =
 U_R(\mathfrak g[t]),
\]
where the right hand side is the current algebra
part of Garland's integral form
(Proposition~\ref{prop:associated-graded}).
Thus $\mathcal Y_R(\mathfrak g)$ is an integral
filtered deformation of $U_R(\mathfrak g[t])$.

We next specialize to an algebraically closed field
$\kk$ of characteristic $p>3$.
In type $A$, Goodwin and Topley \cite{GT} defined
restricted shifted Yangians and their truncations
using the $p$-center.
Our integral form provides a realization of the
restricted Yangian for each classical type.
We call
\[
 \mathscr Y_\kk(\mathfrak g)
 =
 \kk\otimes_R\mathcal Y_R(\mathfrak g)
\]
the hyper Yangian, and define the small Yangian
$y_\kk(\mathfrak g)$ to be its subalgebra generated
by the images of the Drinfeld generators. We prove that it is isomorphic to the restricted
Yangian $Y_\kk^{[p]}(\mathfrak g)$, defined as the
quotient of the modular Yangian by the ideal
generated by the $p$-th powers of the positive
and negative Lyndon root vectors and explicit
central Cartan elements.
This extends the identification in \cite{CHH}
from the positive part to the full Yangian
(Theorem~\ref{thm:small-restricted}).

The integral form  provides a way to study
finite dimensional representations in positive characteristic. For the restricted Yangian of $\mathfrak{gl}_2$,
Chang, Hu, and Topley \cite{CHT} established a
classification of finite dimensional simple modules
in terms of highest weights and Drinfeld polynomials.
Their construction of finite dimensional modules
relies on evaluation modules, tensor products,
and twisting automorphisms, which do not readily
extend to arbitrary classical types. Our strategy is to construct
a finite free stable lattice in the corresponding 
characteristic zero module through  the integral form $\mathcal Y_R(\mathfrak g)$. We lift the
Drinfeld polynomials to a discrete valuation ring
and consider the corresponding finite dimensional
simple modules for Yangian over its fraction field. 
A nondegenerate bilinear form (Lemma \ref{lem:rep-contravariant-form}) 
allows us to prove that the cyclic submodule
generated by a highest vector under the integral
form is a finite free stable lattice
(Proposition~\ref{prop:rep-stable-lattice}).
Reducing this lattice yields a finite dimensional
highest weight module for $\mathscr Y_\kk(\mathfrak g)$ (Corollary \ref{cor:rep-finite-realization}). We then use these modules to study the finite dimensional
simple representations of the small Yangian. We prove that every finite dimensional simple
$y_\kk(\mathfrak g)$-module is highest weight
and classify these modules by $p$-reduced  Drinfeld polynomials
(Theorem~\ref{thm:rep-restricted-classification}). The isomorphism in Theorem~\ref{thm:small-restricted}
gives the same classification for restricted Yangian 
$Y_\kk^{[p]}(\mathfrak g)$. This generalizes \cite[Theorem~3.11]{CHT}   to all classical types.

The paper is organized as follows.
In Section~\ref{sec:preliminaries}, we recall Drinfeld's
presentation of Yangian and  introduce the Cartan series and
formal automorphisms  used later.
In Section~\ref{sec:rank-one}, we prove the basic
divided power identity and establish the triangular
decomposition and PBW theorem for the integral form
of $Y(\mathfrak{sl}_2)$. 
In Section~\ref{sec:classical}, we extend the rank one results
to all classical types and identify the associated
graded algebra of our integral forms with the current algebra part of
Garland's integral form.
In Section~\ref{sec:small-yangian}, we prove that
the small Yangian is isomorphic to the restricted
Yangian. In Section~\ref{sec:rep-lattices}, we construct
stable lattices in finite dimensional characteristic zero   highest weight modules and reduce them modulo $p$
to obtain highest weight modules for the hyper Yangian.
In Section~\ref{sec:restricted-representations},
we classify finite dimensional simple modules
of the small and restricted Yangians by tuples
of $p$-reduced Drinfeld polynomials and discuss
further questions.

\section{Preliminaries}
\label{sec:preliminaries}

\subsection{Drinfeld Yangians and divided powers of root vectors}
\label{subsec:yangian-pbw}

Let $\mathfrak g$ be a classical simple Lie algebra of rank $n$ over
$\mathbb C$, with root system $\Delta$ and Cartan matrix
$(a_{ij})_{i,j\in I}$, where $I=\{1,\ldots,n\}$. We use the
notations of \cite[Section~2.1]{CHH}. Thus we fix 
$d_i=1$ for all $i$ in types $A$ and $D$, whereas
\[
(d_1,\ldots,d_n)
=
\begin{cases}
(1,\ldots,1,\frac12),&\text{in type }B,\\
(1,\ldots,1,2),&\text{in type }C.
\end{cases}
\]
In particular,
$d_i a_{ij}=d_j a_{ji}$. Following Drinfeld \cite{D}, the Yangian $Y(\mathfrak g)$ is the
$\mathbb C$-algebra generated by $\{e_{i,r}, f_{i,r}, h_{i,r}\ |\ i\in I,\ r\in\mathbb N\}$. Put $x_{i,r}^+=e_{i,r}$ and $x_{i,r}^-=f_{i,r}$. Then the defining
relations are
\begin{align}
[h_{i,r},h_{j,s}]&=0,
\label{eq:prelim-hh}\\
[e_{i,r},f_{j,s}]&=\delta_{ij}h_{i,r+s},
\label{eq:prelim-ef}\\
[h_{i,0},x_{j,s}^{\pm}]&=\pm a_{ij}x_{j,s}^{\pm},
\label{eq:prelim-h0x}\\
[h_{i,r+1},x_{j,s}^{\pm}]
-[h_{i,r},x_{j,s+1}^{\pm}]
&=
\pm\frac{d_i a_{ij}}{2}
\bigl(h_{i,r}x_{j,s}^{\pm}
+x_{j,s}^{\pm}h_{i,r}\bigr),
\label{eq:prelim-hx}\\
[x_{i,r+1}^{\pm},x_{j,s}^{\pm}]
-[x_{i,r}^{\pm},x_{j,s+1}^{\pm}]
&=
\pm\frac{d_i a_{ij}}{2}
\bigl(x_{i,r}^{\pm}x_{j,s}^{\pm}
+x_{j,s}^{\pm}x_{i,r}^{\pm}\bigr),
\label{eq:prelim-xx}
\end{align}
together with the Serre relations
\begin{equation}\label{eq:prelim-serre}
\sum_{\sigma\in\mathfrak S_m}
[x_{i,r_{\sigma(1)}}^\pm,
 [x_{i,r_{\sigma(2)}}^\pm,\ldots,
  [x_{i,r_{\sigma(m)}}^\pm,x_{j,s}^\pm]\cdots]]
=0,
\qquad
m=1-a_{ij},\quad i\ne j.
\end{equation}
Let $Y^{>}(\mathfrak g)$, $Y^{<}(\mathfrak g)$, and
$Y^0(\mathfrak g)$ be the subalgebras generated by the $e_{i,r}$,
the $f_{i,r}$, and the $h_{i,r}$, respectively. The PBW theorem for
$Y(\mathfrak g)$ (cf. \cite{Lev,GTL1}) gives the isomorphism
\begin{equation}\label{eq:prelim-rational-triangular}
Y^{<}(\mathfrak g)\otimes_{\mathbb C}Y^0(\mathfrak g)
\otimes_{\mathbb C}Y^{>}(\mathfrak g)
\xrightarrow{\;\sim\;}Y(\mathfrak g).
\end{equation}
Set $R=\mathbb{Z}[1/2]$. Define the $R$-subalgebra
\begin{equation}\label{eq:prelim-full-integral-form}
\mathcal Y_R(\mathfrak g)
=
R\langle e_{i,r}^{(m)},f_{i,r}^{(m)}
\mid i\in I,\ r,m\in\mathbb N\rangle
\subset Y(\mathfrak g),
\end{equation}
and denote its positive and negative part by $\mathcal Y_R^{>}(\mathfrak g)$ and $\mathcal Y_R^{<}(\mathfrak g)$. The assignments
\begin{equation}\label{eq:prelim-tau}
\tau(e_{i,r})=f_{i,r},\qquad
\tau(f_{i,r})=e_{i,r},\qquad
\tau(h_{i,r})=h_{i,r}
\end{equation}
extend to an involutive anti-automorphism of $Y(\mathfrak g)$.
We use the Lyndon root vectors $e_\beta(r)$ fixed in
\cite[(2.15)--(2.17)]{CHH}, and set
\[
f_\beta(r)=\tau(e_\beta(r)).
\]
Order $\Delta^+$ by the lexicographical order on the standard Lyndon
words and $\Delta^+\times\mathbb N$ by
\begin{equation}\label{eq:prelim-root-order}
(\beta,r)<(\gamma,s)
\quad\Longleftrightarrow\quad
\beta<\gamma
\quad\text{or}\quad
\beta=\gamma,\ r>s.
\end{equation}
An arrow $\rightarrow$ denotes multiplication in increasing order
with respect to \eqref{eq:prelim-root-order}, and $\leftarrow$
denotes the reverse order. Following \cite[Theorem~3.8]{CHH},  and using the anti-automorphism $\tau$, we have

\begin{theorem}\label{thm:prelim-one-sided-pbw}
For every $\beta\in\Delta^+$ and $r,m\in\mathbb N$,
\[
e_\beta(r)^{(m)}\in\mathcal Y_R^{>}(\mathfrak g),
\qquad
f_\beta(r)^{(m)}\in\mathcal Y_R^{<}(\mathfrak g).
\]
Moreover, the ordered products
\begin{equation}\label{eq:prelim-one-sided-bases}
\prod_{\beta,r}^{\rightarrow}
e_\beta(r)^{(a_{\beta,r})},
\qquad
\prod_{\beta,r}^{\leftarrow}
f_\beta(r)^{(a_{\beta,r})},
\end{equation}
where $(a_{\beta,r})\in
\mathbb N^{\Delta^+\times\mathbb N}$ has finite support, form
$R$-bases of $\mathcal Y_R^{>}(\mathfrak g)$ and
$\mathcal Y_R^{<}(\mathfrak g)$, respectively.
\end{theorem}

\subsection{The Cartan generators}
\label{subsec:prelim-cartan}

For the rest of this section, let $Y=Y(\mathfrak{sl}_2)$. We write
\[
Y^{>},\quad Y^{<},\quad Y^0,
\qquad
\mathcal Y_R^{>},\quad\mathcal Y_R^{<},\quad\mathcal Y_R
\]
for the corresponding rank one algebras. The following Cartan series is analogous to the one used by
Chari--Pressley in the quantum affine algebra case
\cite[Section~3]{CP}. Put
\begin{equation*}%\label{eq:prelim-H}
H(u)=1+\sum_{r\geq0}h_r u^{-r-1}.
\end{equation*}
Define $Q(u)=\sum_{m\geq0}P_m u^{-m}$ by  $P_0=1$ and 
\begin{equation}\label{eq:rank-difference}
(1+u^{-1})^{h_0}Q(u+1)=H(u)Q(u),
\end{equation}
where
\[
(1+u^{-1})^{h_0}
=
\sum_{m\geq0}\binom{h_0}{m}u^{-m},
\qquad
\binom{x}{m}
=
\frac{x(x-1)\cdots(x-m+1)}{m!}.
\]
Comparing the coefficients of $u^{-m-1}$ in
\eqref{eq:rank-difference} gives
\begin{equation}\label{eq:prelim-P-recursion}
mP_m
=
\sum_{j=0}^{m-1}
\left(
\binom{h_0-j}{m+1-j}-h_{m-j}
\right)P_j
\qquad(m\geq1).
\end{equation}
Hence \eqref{eq:rank-difference} determines the $P_m$ uniquely and,
by induction on $m$,
\begin{equation}\label{eq:prelim-P-triangular}
P_m
=
-\frac{h_m}{m}
+p_m(h_0,\ldots,h_{m-1}),
\qquad
p_m\in\mathbb Q[h_0,\ldots,h_{m-1}].
\end{equation}
For an indeterminate $x$, set
\[
\Bin_{\mathbb Z}(x)
=
\bigoplus_{b\geq0}\mathbb Z\binom{x}{b},
\qquad
\Bin_R(x)
=
R\otimes_{\mathbb Z}\Bin_{\mathbb Z}(x).
\]
These are subrings of $\mathbb Q[x]$. Vandermonde's identity gives
\begin{equation}\label{eq:prelim-binomial-translation}
\binom{x+c}{m}
=
\sum_{j=0}^m
\binom{c}{m-j}\binom{x}{j}
\qquad(c\in\mathbb Z),
\end{equation}
so integer translations preserve both rings. Define
\[
\mathcal H_R
=
\Bin_R(h_0)[P_1,P_2,\ldots]
\subset Y^0.
\]

\begin{lemma}\label{lem:prelim-cartan-basis}
The elements
\begin{equation}\label{eq:prelim-cartan-basis}
\binom{h_0}{b}\prod_{m\geq1}P_m^{c_m},
\qquad
b,c_m\in\mathbb N,
\quad
c_m=0\text{ for all but finitely many }m,
\end{equation}
form an $R$-basis of $\mathcal H_R$.
\end{lemma}

\begin{proof}
By definition, the elements \eqref{eq:prelim-cartan-basis} span $\mathcal H_R$. By the PBW theorem of $Y(\mathfrak g)$,
\[
Y^0=\mathbb C[h_0,h_1,h_2,\ldots].
\]
The relation \eqref{eq:prelim-P-triangular} therefore gives
\[
\mathbb C[h_0,h_1,h_2,\ldots]
=
\mathbb C[h_0,P_1,P_2,\ldots].
\]
In particular, $h_0,P_1,P_2,\ldots$ are algebraically independent.
Since $\left\{\binom{h_0}{b}\mid b\geq0\right\}$ is an $R$-basis of $\Bin_R(h_0)$, the elements in
\eqref{eq:prelim-cartan-basis} are linearly independent over $R$.
\end{proof}

We also need the following scalar series. Set $\beta_0(x)=1$. For
$m\geq1$, let $\beta_m(x)\in\mathbb Q[x]$ be the unique polynomial
satisfying
\begin{equation}\label{eq:rank-beta-recurrence}
\beta_m(0)=0,
\qquad
\beta_m(x+1)-\beta_m(x)
=
x\beta_{m-1}(x),
\end{equation}
and let $B_x(u)
=
1+\sum_{m\geq1}\beta_m(x)u^{-m}$.

\begin{lemma}\label{lem:rank-B-properties}
For every $m\geq0$,
\begin{equation}\label{eq:rank-beta-binomial}
\beta_m(x)\in\Bin_{\mathbb Z}(x).
\end{equation}
The following formulas also hold:
\begin{align}
B_{x+1}(u)
&=
\left(1+\frac{x}{u}\right)B_x(u),\label{eq:rank-B-shifts-1}\\
B_x(u+1)
&=
(1+u^{-1})^{-x}
\left(1+\frac{x}{u}\right)B_x(u)\label{eq:rank-B-shifts-2}.
\end{align}
\end{lemma}

\begin{proof}
Assume inductively that
$\beta_{m-1}(x)\in\Bin_{\mathbb Z}(x)$.
The identity
\[
 x\binom{x}{j}
 =
 j\binom{x}{j}+(j+1)\binom{x}{j+1}
\]
shows that we can write
\[
 x\beta_{m-1}(x)=\sum_{j\geq0}c_j\binom{x}{j},
 \qquad c_j\in\mathbb Z,
\]
with only finitely many nonzero $c_j$. The polynomial $\sum_{j\geq0}c_j\binom{x}{j+1}$
satisfies the defining recurrence for $\beta_m(x)$
and vanishes at $x=0$.
By uniqueness,
\[
 \beta_m(x)=\sum_{j\geq0}c_j\binom{x}{j+1}
 \in\Bin_{\mathbb Z}(x).
\]
This proves \eqref{eq:rank-beta-binomial}. Summing \eqref{eq:rank-beta-recurrence} over $m\geq1$ gives
\[
B_{x+1}(u)
=
\left(1+\frac{x}{u}\right)B_x(u),
\]
 For $x=N\in\mathbb N$, the product
formula gives
\[
B_N(u+1)
=
(1+u^{-1})^{-N}
\left(1+\frac{N}{u}\right)B_N(u).
\]
For each power of $u^{-1}$, both sides are polynomials in $x$.
Since they agree for every $x=N\in\mathbb N$, they agree identically,
and \eqref{eq:rank-B-shifts-2} follows.
\end{proof}

Since $B_x(u)$ has constant coefficient $1$, it has a unique inverse
\begin{equation}\label{eq:inverseBx}
B_x(u)^{-1}
=
\sum_{m\geq0}\gamma_m(x)u^{-m},
\qquad
\gamma_0(x)=1.
\end{equation}
The coefficients satisfy
\[
\gamma_m(x)
=
-\sum_{j=1}^m
\beta_j(x)\gamma_{m-j}(x)
\qquad(m\geq1).
\]
Since $\Bin_R(x)$ is a ring, \eqref{eq:rank-beta-binomial} implies
\begin{equation}\label{eq:rank-gamma-binomial}
\gamma_m(x)\in\Bin_R(x)
\qquad(m\geq0).
\end{equation}

\subsection{Automorphisms of $Y[[z]]$}
\label{subsec:prelim-automorphisms}
Let \(z\) be a central indeterminate and set
\[
Y[[z]]
=
\left\{
 \sum_{n\geq0}a_nz^n
 \,\middle|\,
 a_n\in Y
\right\},
\]
which is the algebra of formal power series in $z$ with coefficients in $Y$.
The  positive and negative subalgebras of  $Y[[z]]$ are denoted by 
$Y^{>}[[z]]$ and $Y^{<}[[z]]$.
Likewise, define
\[
 \mathcal Y_R^{>}[[z]]
 =
 \left\{\sum_{n\geq0}a_nz^n
          \,\middle|\,a_n\in\mathcal Y_R^{>}\right\},
\]
and similarly for $\mathcal Y_R^{<}[[z]]$.
They are $R[[z]]$-subalgebras of $Y[[z]]$.

Let $S$ be the shift operator on the Drinfeld generators,
\[
Se_r=e_{r+1},\qquad
Sf_r=f_{r+1},\qquad
Sh_r=h_{r+1}.
\]
This notation is related to the shift operators appearing in
\cite[Section~2.6]{GTL1}. We shall also use the conjugation formulas
\cite[Lemma~2.7]{GTL1}
\begin{equation}\label{eq:prelim-H-conjugation}
H(u)e_rH(u)^{-1}
=
\frac{u-S+1}{u-S-1}e_r,
\qquad
H(u)f_rH(u)^{-1}
=
\frac{u-S-1}{u-S+1}f_r,
\end{equation}
where the rational functions in $u$ are expanded in powers of
$u^{-1}$.

The automorphisms $\Phi_{c,z}$ in the following proposition can be
obtained from 
\cite[Proposition~4.1]{Zhang}. We give a direct proof and show, moreover, that these
automorphisms preserve the positive and negative integral forms.

\begin{proposition}\label{prop:prelim-formal-automorphisms}
For every $c\in R$, the assignments
\begin{equation}\label{eq:prelim-Phi}
\begin{aligned}
\Phi_{c,z}(e_r)
&=
(1-z(S+c))e_r
=
(1-cz)e_r-ze_{r+1},\\
\Phi_{c,z}(f_r)
&=
(1-z(S+c))^{-1}f_r,\\
\Phi_{c,z}(h_r)
&=
h_r
\end{aligned}
\end{equation}
define a $\mathbb C[[z]]$-algebra automorphism of $Y[[z]]$. Here
\begin{equation}\label{eq:prelim-inverse-mode-factor}
(1-z(S+c))^{-1}x_r
=
\sum_{n\geq0}z^n(S+c)^n x_r
\qquad(x_r=e_r\text{ or }f_r).
\end{equation}
Its inverse is given by
\begin{equation}\label{eq:prelim-Phi-inverse}
\begin{aligned}
\Phi_{c,z}^{-1}(e_r)
&=
(1-z(S+c))^{-1}e_r,\\
\Phi_{c,z}^{-1}(f_r)
&=
(1-z(S+c))f_r,\\
\Phi_{c,z}^{-1}(h_r)
&=
h_r.
\end{aligned}
\end{equation}
Moreover,
\begin{equation}\label{eq:prelim-Phi-integral}
\Phi_{c,z}\bigl(\mathcal Y_R^{>}[[z]]\bigr)
=
\mathcal Y_R^{>}[[z]],
\qquad
\Phi_{c,z}\bigl(\mathcal Y_R^{<}[[z]]\bigr)
=
\mathcal Y_R^{<}[[z]].
\end{equation}
\end{proposition}

\begin{proof}
We first  show that $\Phi_{c,z}$ defines a $R[[z]]$-algebra homomorphism for $Y^>[[z]]$ and $Y^<[[z]]$. For
$\varepsilon\in\{1,-1\}$, let $x_r=e_r$ if $\varepsilon=1$, and
$x_r=f_r$ if $\varepsilon=-1$. Set
\[
T_{r,s}^{\varepsilon}(x)
=
[x_{r+1},x_s]-[x_r,x_{s+1}]
-\varepsilon(x_rx_s+x_sx_r).
\]
 If $a,b$ are central scalars and $y_r=ax_r+bx_{r+1}$, then 
\[
T_{r,s}^{\varepsilon}(y)
=
a^2T_{r,s}^{\varepsilon}(x)
+ab\bigl(
T_{r+1,s}^{\varepsilon}(x)
+T_{r,s+1}^{\varepsilon}(x)
\bigr)
+b^2T_{r+1,s+1}^{\varepsilon}(x).
\]
Hence the substitution $x_r\mapsto ax_r+bx_{r+1}$ preserves the defining relations \eqref{eq:prelim-xx}.

Taking $r=s$ in \eqref{eq:prelim-xx} gives $[x_{r+1},x_r]=\varepsilon x_r^2$. An induction on $n$ then yields
\begin{equation}\label{eq:prelim-adjacent-binomial}
(a x_r+b x_{r+1})^{(n)}
=
\sum_{j=0}^{n}
b^{n-j}
\left(
\prod_{q=0}^{j-1}(a+\varepsilon q b)
\right)
x_r^{(j)}x_{r+1}^{(n-j)}.
\end{equation}

For the positive part, take $\varepsilon=1$,
$a=1-cz$, and $b=-z$. Every coefficient on the right hand side of
\eqref{eq:prelim-adjacent-binomial} belongs to $R[z]$. Thus
\[
e_r\longmapsto(1-z(S+c))e_r
\]
defines a homomorphism of $Y^{>}[[z]]$, carrying
$\mathcal Y_R^{>}[[z]]$ into itself.  On the generators, its inverse is given by
$(1-z(S+c))^{-1}$, which also preserves $\mathcal Y_R^{>}[[z]]$. Similarly, 
\[
f_r\longmapsto(1-z(S+c))f_r.
\]
defines a homomorphism of $Y^{<}[[z]]$, carrying
$\mathcal Y_R^{<}[[z]]$ into itself, and its  inverse is given by  $f_r\longmapsto(1-z(S+c))^{-1}f_r$, which also preserves $\mathcal Y_R^{<}[[z]]$.

It remains to check the relations involving both positive and
negative generators. Relations \eqref{eq:prelim-h0x} and
\eqref{eq:prelim-hx} are linear in the root generators and remain
valid after applying a power series in $S+c$. Moreover,
\eqref{eq:prelim-ef} and Vandermonde's identity give
\begin{equation}\label{eq:prelim-shifted-mixed}
[(S+c)^a e_r,(S+c)^b f_s]
=
(S+c)^{a+b}h_{r+s}
\qquad(a,b\in\mathbb N).
\end{equation}
Therefore
\[
[\Phi_{c,z}(e_r),\Phi_{c,z}(f_s)]=
(1-z(S+c))(1-z(S+c))^{-1}h_{r+s}=h_{r+s}.
\]
Together with \eqref{eq:prelim-hh}, this proves that
\eqref{eq:prelim-Phi} defines a $R[[z]]$-algebra  homomorphism of $Y[[z]]$.
The formulas in \eqref{eq:prelim-Phi-inverse} give its inverse.
\end{proof}

Set $E(u)=\sum_{r\geq0}e_ru^{-r}$, $F(u)=\sum_{r\geq0}f_{r+1}u^{-r}$.

\begin{corollary}\label{cor:prelim-current-integrality}
For all $a,b,n\in\mathbb N$,
\begin{equation}\label{eq:current-integrality}
[u^{-n}]E(u)^{(a)}
\in\mathcal Y_R^{>},
\qquad
[u^{-n}]F(u)^{(b)}
\in\mathcal Y_R^{<},
\end{equation}
where $[u^{-n}]$ denotes the coefficient of $u^{-n}$
\end{corollary}

\begin{proof}
By \eqref{eq:prelim-Phi} and \eqref{eq:prelim-Phi-inverse},
\[
\Phi_{0,u^{-1}}^{-1}(e_0)=E(u),
\qquad
\Phi_{0,u^{-1}}(f_1)=F(u).
\]
Hence \eqref{eq:prelim-Phi-integral} gives
\[
E(u)^{(a)}
=
\Phi_{0,u^{-1}}^{-1}\bigl(e_0^{(a)}\bigr)
\in\mathcal Y_R^{>}[[u^{-1}]],\qquad F(u)^{(b)}
=
\Phi_{0,u^{-1}}\bigl(f_1^{(b)}\bigr)
\in\mathcal Y_R^{<}[[u^{-1}]].
\]
Taking coefficients proves \eqref{eq:current-integrality}.
\end{proof}

\section{The integral form in rank one}
\label{sec:rank-one}

Throughout this section, let $Y=Y(\mathfrak{sl}_2)$, and use the notation introduced in Section~\ref{sec:preliminaries}.
Our goal is to identify the $R$-subalgebra $\mathcal{Y}_R$ generated by the positive
and negative divided powers. We first introduce the triangular span
built from the positive, negative, and Cartan parts,  and establish the
commutation relations needed to work with it. We then prove an
analogue of the basic lemma of Chari--Pressley \cite[Lemma~5.1]{CP}.
Starting from the resulting formula for $e_0^{(a)}f_1^{(b)}$, we
extend the argument to products $e_r^{(a)}f_s^{(b)}$ for arbitrary
$r,s\in\mathbb N$, and finally obtain the PBW theorem for $\mathcal{Y}_R$.

\subsection{Commutation relations with Cartan generators}
\label{subsec:rank-triangular-span}

Let
\[
M=\mathcal Y_R^{<}\mathcal H_R\mathcal Y_R^{>}
\subset Y
\]
denote the $R$-linear span of all products $x^-hx^+$ with
$x^-\in\mathcal Y_R^{<}$, $h\in\mathcal H_R$, and
$x^+\in\mathcal Y_R^{>}$. By
Theorem~\ref{thm:prelim-one-sided-pbw},
\eqref{eq:prelim-rational-triangular}, and
\eqref{eq:prelim-cartan-basis}, the products
\begin{equation}\label{eq:rank-basis}
\left(\prod_{r\geq0}^{\leftarrow}f_r^{(a_r)}\right)
\binom{h_0}{b}
\left(\prod_{m\geq1}P_m^{c_m}\right)
\left(\prod_{r\geq0}^{\rightarrow}e_r^{(d_r)}\right)
\end{equation}
form an $R$-basis of $M$, where all exponent families have finite
support.

\begin{lemma}\label{lem:rank-Q-conjugation}
For every $r\geq0$,
\begin{equation}\label{eq:rank-Q-conjugation}
\begin{aligned}
Q(u)e_rQ(u)^{-1}
&=\frac{(u-S)(u-S-1)}{u^2}e_r,\\
Q(u)^{-1}f_rQ(u)
&=\frac{(u-S)(u-S-1)}{u^2}f_r.
\end{aligned}
\end{equation}
\end{lemma}

\begin{proof}
Let  $G(u)=(1+u^{-1})^{-h_0}H(u)=1+O(u^{-2})$. The defining equation for $Q(u)$ gives $Q(u+1)=G(u)Q(u)$.
Consequently, $X(u)=Q(u)e_rQ(u)^{-1}$
satisfies
\begin{equation}\label{eq:rank-conjugation-difference}
X(u+1)=G(u)X(u)G(u)^{-1},
\qquad
X(u)=e_r+O(u^{-1}).
\end{equation}
This equation has at most one solution with the stated constant term.
Indeed, suppose that the difference of two solutions begins with
$xu^{-m}$, where $m\geq1$ and $x\ne0$. Its change under
$u\mapsto u+1$ begins with $-mxu^{-m-1}$, whereas conjugation by
$G(u)=1+O(u^{-2})$ changes it only in order $u^{-m-2}$, a
contradiction. Set
\[
K(u,S)=\frac{(u-S)(u-S-1)}{u^2},
\qquad
\widetilde X(u)=K(u,S)e_r.
\]
Explicitly, $\widetilde X(u)
=e_r-(e_r+2e_{r+1})u^{-1}
 +(e_{r+1}+e_{r+2})u^{-2}$. Since $[h_0,e_j]=2e_j$ for all $j\geq0$, we have
$[h_0,\widetilde X(u)]=2\widetilde X(u)$. If
$[h_0,x]=2x$, then
\[
p(h_0)x=xp(h_0+2)
\]
for every polynomial $p$. Applying this coefficientwise to $(1+u^{-1})^{-h_0}$ gives
\[
(1+u^{-1})^{-h_0}x(1+u^{-1})^{h_0}
=(1+u^{-1})^{-2}x.
\]
On the other hand, \eqref{eq:prelim-H-conjugation} gives $H(u)\widetilde X(u)H(u)^{-1}
=K(u,S)\frac{u-S+1}{u-S-1}e_r$. It follows that
\[
\begin{aligned}
G(u)\widetilde X(u)G(u)^{-1}
&=(1+u^{-1})^{-2}
K(u,S)\frac{u-S+1}{u-S-1}e_r\\
&=K(u+1,S)e_r
=\widetilde X(u+1).
\end{aligned}
\]
Thus $\widetilde X(u)$ satisfies
\eqref{eq:rank-conjugation-difference}, and uniqueness gives the
first identity in \eqref{eq:rank-Q-conjugation}. Applying the
anti-automorphism $\tau$ gives the second identity.
\end{proof}

\begin{proposition}\label{prop:rank-cartan-exchange}
One has
\[
\mathcal H_R\mathcal Y_R^{>}
=\mathcal Y_R^{>}\mathcal H_R,
\qquad
\mathcal H_R\mathcal Y_R^{<}
=\mathcal Y_R^{<}\mathcal H_R.
\]
\end{proposition}

\begin{proof}
For the binomial generators, the weight relations give
\[
\binom{h_0}{b}e_r^{(a)}
=e_r^{(a)}\binom{h_0+2a}{b},
\qquad
\binom{h_0}{b}f_r^{(a)}
=f_r^{(a)}\binom{h_0-2a}{b}.
\]
The translated binomials belong to $\mathcal H_R$ by
\eqref{eq:prelim-binomial-translation}. We next consider the generators $P_m$. The operator in
\eqref{eq:rank-Q-conjugation} factors as
\[
K(u,S)
=\bigl(1-u^{-1}S\bigr)
 \bigl(1-u^{-1}(S+1)\bigr).
\]
After setting $z=u^{-1}$, each factor is of the form
$1-z(S+c)$ with $c\in\{0,1\}$. Proposition~\ref{prop:prelim-formal-automorphisms}
therefore shows that both $K(u,S)$ and $K(u,S)^{-1}$ preserve
$\mathcal Y_R^{>}[[u^{-1}]]$ and
$\mathcal Y_R^{<}[[u^{-1}]]$. Since conjugation is an algebra
automorphism, Lemma~\ref{lem:rank-Q-conjugation} implies that
conjugation by $Q(u)$ and by $Q(u)^{-1}$ preserves $\mathcal Y_R^{>}[[u^{-1}]]$ and
$\mathcal Y_R^{<}[[u^{-1}]]$. For $x\in\mathcal Y_R^{>}$, write
\[
Q(u)xQ(u)^{-1}=\sum_{i\geq0}x_i u^{-i},
\qquad
Q(u)^{-1}xQ(u)=\sum_{i\geq0}y_i u^{-i},
\]
where $x_i,y_i\in\mathcal Y_R^{>}$ and $x_0=y_0=x$. Then
\[
Q(u)x=\left(\sum_{i\geq0}x_i u^{-i}\right)Q(u),
\qquad
xQ(u)=Q(u)\left(\sum_{i\geq0}y_i u^{-i}\right).
\]
Comparing the coefficients of $u^{-m}$ gives
\[
P_mx=\sum_{i=0}^m x_iP_{m-i},
\qquad
xP_m=\sum_{i=0}^mP_{m-i}y_i.
\]
Hence
$\mathcal H_R\mathcal Y_R^{>}
=\mathcal Y_R^{>}\mathcal H_R$.
The same argument  gives
$\mathcal H_R\mathcal Y_R^{<}
=\mathcal Y_R^{<}\mathcal H_R$.
\end{proof}

It follows that
\begin{equation}\label{eq:rank-M-stability}
\mathcal Y_R^{<}M\subseteq M,
\qquad
M\mathcal Y_R^{>}\subseteq M,
\qquad
\mathcal H_RM\subseteq M,
\qquad
M\mathcal H_R\subseteq M.
\end{equation}
We shall use these inclusions before proving that $M$ is an algebra.

The automorphisms of Section~\ref{subsec:prelim-automorphisms} also
preserve $M$:
\begin{equation}\label{eq:rank-Phi-M}
\Phi_{c,z}(M[[z]])=M[[z]]
\qquad(c\in R).
\end{equation}
Indeed, $\Phi_{c,z}$ fixes $\mathcal H_R$ pointwise and preserves
$\mathcal Y_R^{>}[[z]]$ and $\mathcal Y_R^{<}[[z]]$. Thus each
coefficient in the image of a product $x^-hx^+$ lies in $M$. This
gives one inclusion in \eqref{eq:rank-Phi-M}, and the inverse
automorphism gives the other.

\subsection{The basic lemma}
\label{subsec:rank-basic}

The following formula for $e_0^{(a)}f_1^{(b)}$ is the Yangian
analogue of the basic lemmas of Garland \cite[Lemma~ 7.5]{Gar}  and Chari--Pressley
\cite[Lemma~5.1]{CP}.

\begin{theorem}\label{thm:rank-basic}
For all $a,b\in\mathbb N$,
\begin{equation}\label{eq:rank-basic}
e_0^{(a)}f_1^{(b)}
=
\sum_{t=0}^{\min(a,b)}
(-1)^t[u^{-t}]
\left(
F(u)^{(b-t)}
\frac{Q(u)}{B_{h_0-a-b+t+1}(u)}
E(u)^{(a-t)}
\right).
\end{equation}
Every summand on the right hand side belongs to $M$.
\end{theorem}

For $c\in\mathbb Z$, set $\mathcal A_c(u)
=u^c\frac{Q(u)}{B_{h_0-c}(u)}$. By \eqref{eq:inverseBx},
\begin{equation}\label{eq:rank-Ac-triangular}
u^{-c}\mathcal A_c(u)
=Q(u)B_{h_0-c}(u)^{-1}
=\sum_{m\geq0}
\left(
\sum_{j=0}^mP_{m-j}\gamma_j(h_0-c)
\right)u^{-m}
\in\mathcal H_R[[u^{-1}]].
\end{equation}

\begin{lemma}\label{lem:rank-Ac-shifts}
For every $c\in\mathbb Z$,
\begin{align}
\mathcal A_{c+1}(u)
&=(u+h_0-c-1)\mathcal A_c(u),\label{eq:rank-Ac-shifts-1}\\
\mathcal A_{c+1}(u+1)
&=uH(u)\mathcal A_c(u).\label{eq:rank-Ac-shifts-2}
\end{align}
\end{lemma}

\begin{proof}
Put $d=h_0-c-1$. \eqref{eq:rank-B-shifts-1} gives $B_{h_0-c}(u)=B_{d+1}(u)
=\left(1+\frac{d}{u}\right)B_d(u)$. Hence
\[
\begin{aligned}
\mathcal A_{c+1}(u)=u^{c+1}\frac{Q(u)}{B_d(u)}=(u+d)u^c\frac{Q(u)}{B_{d+1}(u)}=(u+h_0-c-1)\mathcal A_c(u).
\end{aligned}
\]
Using 
\eqref{eq:rank-B-shifts-2}, we obtain
\[
\begin{aligned}
\mathcal A_{c+1}(u+1)
&=(u+1)^{c+1}\frac{Q(u+1)}{B_d(u+1)}\\
&=(u+1)^{c+1}(1+u^{-1})^{-c-1}
H(u)\frac{Q(u)}{B_{d+1}(u)}\\
&=u^{c+1}H(u)\frac{Q(u)}{B_{h_0-c}(u)}=uH(u)\mathcal A_c(u).
\end{aligned}
\]
This proves \eqref{eq:rank-Ac-shifts-2}.
\end{proof}

Set
\[
\varepsilon(u)=u^{-1}E(u)
=\sum_{r\geq0}e_ru^{-r-1},
\qquad
\varphi(u)=u^{-1}F(u)
=\sum_{r\geq0}f_{r+1}u^{-r-1}.
\]
The rank-one relations \eqref{eq:prelim-xx} are equivalently written
as
\begin{equation}\label{eq:rank-normalized-current-relations}
\begin{aligned}
(u-v)[\varepsilon(u),\varepsilon(v)]
&=-\bigl(\varepsilon(u)-\varepsilon(v)\bigr)^2,\\
(u-v)[\varphi(u),\varphi(v)]
&=\bigl(\varphi(u)-\varphi(v)\bigr)^2.
\end{aligned}
\end{equation}

\begin{lemma}\label{lem:rank-residue-identities}
For $k\geq1$,
\[
[e_0,\varphi(u)^{(k)}]
=u\varphi(u+1)^{(k-1)}H(u)
-\varphi(u)^{(k-1)}(u+h_0-k+1).
\]
For every $c\in\mathbb Z$ and $l\geq0$,
\[
e_0\mathcal A_c(u)\varepsilon(u)^{(l)}
=\mathcal A_{c+2}(u)
\left(
\varepsilon(u-1)^{(l+1)}
-\varepsilon(u)^{(l+1)}
\right).
\]
\end{lemma}

\begin{proof}
We first compute the commutator with $\varphi(u)^{(k)}$. Since
$[e_0,f_{r+1}]=h_{r+1}$,
\[
[e_0,\varphi(u)]
=uH(u)-u-h_0.
\]
Moreover, $\varphi(u)=(u-S)^{-1}f_1$, so the second identity in
\eqref{eq:prelim-H-conjugation} gives
\[
H(u)\varphi(u)H(u)^{-1}
=2\varphi(u+1)-\varphi(u).
\]
Put $x=\varphi(u)$, $y=\varphi(u+1)$, and $d=y-x$. Taking
$v=u+1$ in the second identity of
\eqref{eq:rank-normalized-current-relations} gives
$[d,y]=d^2$. Hence, by induction on $m$,
\begin{equation}\label{eq:rank-negative-power}
(y+d)^m=y^m+mdy^{m-1},
\end{equation}
and therefore
\begin{equation}\label{eq:rank-negative-sum}
\sum_{i=0}^{k-1}x^i(2y-x)^{k-1-i}
=ky^{k-1}.
\end{equation}
Using $[e_0,x]=uH(u)-u-h_0$, we obtain
\[
[e_0,x^k]
=\sum_{i=0}^{k-1}x^i(uH(u)-u-h_0)x^{k-1-i}.
\]
Since
$H(u)x^j=(2y-x)^jH(u)$, the terms containing $H(u)$ give, by
\eqref{eq:rank-negative-sum},
\[
u\sum_{i=0}^{k-1}x^iH(u)x^{k-1-i}
=uky^{k-1}H(u).
\]
Since $[h_0,x]=-2x$,
\[
\sum_{i=0}^{k-1}x^i(u+h_0)x^{k-1-i}
=kx^{k-1}(u+h_0-k+1).
\]
Dividing the resulting identity by $k!$ proves the first assertion.

For the second assertion, the inverse of the positive identity in
\eqref{eq:rank-Q-conjugation} gives
\[
Q(u)^{-1}e_0Q(u)
=\frac{u^2}{(u-S)(u-S-1)}e_0.
\]
Let $\eta(u)=\varepsilon(u-1)-\varepsilon(u)$. Then
\[
\eta(u)
=\frac{1}{(u-S)(u-S-1)}e_0,
\]
and hence $e_0Q(u)=Q(u)u^2\eta(u)$. Because $[h_0,\eta(u)]=2\eta(u)$,
\[
\eta(u)B_{h_0-c}(u)^{-1}
=B_{h_0-c-2}(u)^{-1}\eta(u).
\]
Therefore
\[
\begin{aligned}
e_0\mathcal A_c(u)=u^ce_0Q(u)B_{h_0-c}(u)^{-1}=u^{c+2}Q(u)B_{h_0-c-2}(u)^{-1}\eta(u)=\mathcal A_{c+2}(u)\eta(u).
\end{aligned}
\]
Taking $v=u-1$ in the first identity of
\eqref{eq:rank-normalized-current-relations} gives
$[\eta(u),\varepsilon(u)]=\eta(u)^2$. Applying
\eqref{eq:rank-negative-power} with
$y=\varepsilon(u)$ and $d=\eta(u)$ yields $\varepsilon(u-1)^{(l+1)}
-\varepsilon(u)^{(l+1)}
=\eta(u)\varepsilon(u)^{(l)}$.
\end{proof}

\begin{proof}[Proof of Theorem~\ref{thm:rank-basic}]
For a Laurent series at infinity with finitely many positive
powers,
\begin{equation}\label{eq:rank-residue-translation}
[u^{-1}]T(u+1)=[u^{-1}]T(u).
\end{equation}
For $k,l\in\mathbb N$ and $c\in\mathbb Z$, set
\[
\mathcal C_{k,c,l}
=[u^{-1}]
\left(
\varphi(u)^{(k)}\mathcal A_c(u)\varepsilon(u)^{(l)}
\right),
\]
and
\[
Z_{a,b}
=\sum_{t=0}^{\min(a,b)}
(-1)^t\mathcal C_{b-t,a+b-t-1,a-t}.
\]
This is the right hand side of \eqref{eq:rank-basic}. We will prove
\begin{equation}\label{eq:rank-Z-recurrence}
e_0Z_{a,b}=(a+1)Z_{a+1,b}.
\end{equation}
Lemmas~\ref{lem:rank-Ac-shifts} and
\ref{lem:rank-residue-identities}, together with
\eqref{eq:rank-residue-translation}, give
\begin{equation}\label{eq:rank-C-recurrence}
\begin{aligned}
e_0\mathcal C_{k,c,l}
={}&[u^{-1}]
\left(
\varphi(u+1)^{(k-1)}
\mathcal A_{c+1}(u+1)
\varepsilon(u)^{(l)}
\right)\\
&-[u^{-1}]
\left(
\varphi(u)^{(k-1)}
(u+h_0-k+1)\mathcal A_c(u)
\varepsilon(u)^{(l)}
\right)\\
&+[u^{-1}]
\left(
\varphi(u+1)^{(k)}
\mathcal A_{c+2}(u+1)
\varepsilon(u)^{(l+1)}
\right)\\
&-[u^{-1}]
\left(
\varphi(u)^{(k)}
\mathcal A_{c+2}(u)
\varepsilon(u)^{(l+1)}
\right).
\end{aligned}
\end{equation}
Let $m=\min(a,b)$. For $0\leq t\leq m$, set
\[
k_t=b-t,
\qquad
c_t=a+b-t-1,
\qquad
l_t=a-t,
\]
and denote the four terms on the right hand side of
\eqref{eq:rank-C-recurrence} by $I_t,-J_t,K_t,-L_t$,
respectively. If $t=b$, set $I_t=J_t=0$. Then
\[
e_0Z_{a,b}
=\sum_{t=0}^{m}(-1)^t(I_t-J_t+K_t-L_t).
\]
For $0\leq t<m$, \eqref{eq:rank-residue-translation} gives
$I_t=K_{t+1}$, and \eqref{eq:rank-Ac-shifts-1} gives
\[
\begin{aligned}
L_{t+1}-J_t
&=[u^{-1}]
\left(
\varphi(u)^{(k_t-1)}
\bigl(
\mathcal A_{c_t+1}(u)
-(u+h_0-k_t+1)\mathcal A_{c_t}(u)
\bigr)
\varepsilon(u)^{(l_t)}
\right)\\
&=-(a+1)\mathcal C_{k_t-1,c_t,l_t}.
\end{aligned}
\]
Consequently,
\begin{equation}\label{eq:rank-keyformula}
\begin{aligned}
e_0Z_{a,b}
={}&(K_0-L_0)+(-1)^m(I_m-J_m)+(a+1)\sum_{t=0}^{m-1}(-1)^{t+1}
\mathcal C_{k_t-1,c_t,l_t}.
\end{aligned}
\end{equation}
Using \eqref{eq:rank-residue-translation} for $K_0$ and the second
assertion of Lemma~\ref{lem:rank-residue-identities}, we obtain
\[
\begin{aligned}
K_0-L_0
&=[u^{-1}]
\left(
\varphi(u)^{(b)}
\mathcal A_{a+b+1}(u)
\bigl(\varepsilon(u-1)-\varepsilon(u)\bigr)
\varepsilon(u)^{(a)}
\right).
\end{aligned}
\]
Since
\[
\begin{aligned}
\varepsilon(u-1)-\varepsilon(u)
&=e_0u^{-2}+O(u^{-3}),\\
\varepsilon(u)^{(a)}
&=e_0^{(a)}u^{-a}+O(u^{-a-1}),\\
\varphi(u)^{(b)}
&=f_1^{(b)}u^{-b}+O(u^{-b-1}),\\
\mathcal A_{a+b+1}(u)
&=u^{a+b+1}(1+O(u^{-1})),
\end{aligned}
\]
it follows that
\[
K_0-L_0
=(a+1)f_1^{(b)}e_0^{(a+1)}
=(a+1)\mathcal C_{b,a+b,a+1}.
\]
If $b\leq a$, then $m=b$ and $I_m=J_m=0$. If $a<b$,
then $m=a$, and 
\eqref{eq:rank-Ac-shifts-1} gives
\[
I_a-J_a
=-(a+1)\mathcal C_{b-a-1,b-1,0}.
\]
Substituting these identities into \eqref{eq:rank-keyformula}
gives \eqref{eq:rank-Z-recurrence}. For $a=0$, only $t=0$ occurs and
$Z_{0,b}=\mathcal C_{b,b-1,0}$. Since
\[
\varphi(u)^{(b)}
=f_1^{(b)}u^{-b}+O(u^{-b-1}),
\qquad
\mathcal A_{b-1}(u)
=u^{b-1}(1+O(u^{-1})),
\]
we have $Z_{0,b}=f_1^{(b)}$. Induction on $a$ now gives $e_0^{(a)}f_1^{(b)}=Z_{a,b}$, which proves \eqref{eq:rank-basic}.

It remains to verify that every summand on the right hand side of
\eqref{eq:rank-basic} lies in $M$. Fix $t$ and put $C_t(u)=\frac{Q(u)}{B_{h_0-a-b+t+1}(u)}$. By \eqref{eq:rank-Ac-triangular},
$C_t(u)\in\mathcal H_R[[u^{-1}]]$, while
\eqref{eq:current-integrality} gives
\[
[u^{-i}]F(u)^{(b-t)}\in\mathcal Y_R^{<},
\qquad
[u^{-k}]E(u)^{(a-t)}\in\mathcal Y_R^{>}.
\]
Hence the $t$-th summand is the finite sum
\[
\sum_{\substack{i,j,k\geq0\\i+j+k=t}}
\left([u^{-i}]F(u)^{(b-t)}\right)
\left([u^{-j}]C_t(u)\right)
\left([u^{-k}]E(u)^{(a-t)}\right),
\]
and every term belongs to $M$.
\end{proof}

\subsection{Products of divided powers}
\label{subsec:rank-products}

We now pass from the basic formula for $e_0^{(a)}f_1^{(b)}$ to
products $e_r^{(a)}f_s^{(b)}$ with arbitrary $r,s\in\mathbb N$.
For $n,b\in\mathbb N$, put
\[
D_n^-(b)=[u^{-n}]F(u)^{(b)},
\qquad
E_n^+(b)=[u^{-n}]E(u)^{(b)}.
\]
By \eqref{eq:current-integrality}, these elements belong to
$\mathcal Y_R^{<}$ and $\mathcal Y_R^{>}$, respectively.

\begin{lemma}\label{lem:rank-e0-current}
For all $a,b,n\in\mathbb N$,
\[
e_0^{(a)}D_n^-(b)\in M.
\]
\end{lemma}

\begin{proof}
We induct on $n$, simultaneously for all $a,b\in\mathbb N$. The
case $n=0$ follows from $D_0^-(b)=f_1^{(b)}$ and
Theorem~\ref{thm:rank-basic}. Let $n>0$ and assume that
\[
e_0^{(a)}D_i^-(c)\in M
\qquad
(a,c\in\mathbb N,\ 0\leq i<n).
\]
Applying \eqref{eq:rank-basic} with exponents $n$ and $n+b$ gives
\[
e_0^{(n)}f_1^{(n+b)}
=\sum_{t=0}^{n}(-1)^t[u^{-t}]
\left(
F(u)^{(n+b-t)}
\frac{Q(u)}{B_{h_0-2n-b+t+1}(u)}
E(u)^{(n-t)}
\right).
\]
Write
\[
\frac{Q(u)}{B_{h_0-2n-b+t+1}(u)}
=\sum_{j\geq0}c_{t,j}u^{-j},
\qquad
c_{t,0}=1,
\quad c_{t,j}\in\mathcal H_R.
\]
The coefficient in the $t$-th summand is
\[
\sum_{\substack{i,j,l\geq0\\i+j+l=t}}
D_i^-(n+b-t)c_{t,j}E_l^+(n-t).
\]
The term with $t=n$ and $i=n$, $j=l=0$ is
$(-1)^nD_n^-(b)$, while every other term involving a negative
coefficient has index $i<n$. Thus
\begin{equation}\label{eq:rank-D-extraction}
e_0^{(n)}f_1^{(n+b)}
=(-1)^nD_n^-(b)+W,
\qquad
W\in
\sum_{\substack{0\leq i<n\\c\geq0}}
D_i^-(c)\mathcal H_R\mathcal Y_R^{>},
\end{equation}
Multiply \eqref{eq:rank-D-extraction} on the left by $e_0^{(a)}$.
The left hand side belongs to $M$, since
\[
e_0^{(a)}e_0^{(n)}f_1^{(n+b)}
=\binom{a+n}{a}e_0^{(a+n)}f_1^{(n+b)}\in M
\]
by Theorem~\ref{thm:rank-basic}. By the induction hypothesis,
$e_0^{(a)}D_i^-(c)\in M$ for $i<n$, and right multiplication by
the remaining Cartan and positive factors preserves $M$ by
\eqref{eq:rank-M-stability}. Hence $e_0^{(a)}W\in M$, and the
claim follows from \eqref{eq:rank-D-extraction}.
\end{proof}

We use the root grading on $\mathcal Y_R^{<}$ in which each $f_j$
has degree $1$, together with the loop filtration in which $f_j$
has degree $j$. In rank one, the ordering fixed in
\eqref{eq:prelim-root-order} gives ordered products of the form
\[
x=f_0^{(a_0)}f_1^{(a_1)}f_2^{(a_2)}\cdots,
\]
with finitely many nonzero exponents. Set
\[
\deg_{\mathrm{rt}}x=\sum_{j\geq0}a_j,
\qquad
\deg_{\mathrm{loop}}x=\sum_{j\geq0}ja_j.
\]
The following lemma isolates $f_s^{(b)}$ in
$D_{(s-1)b}^-(b)$.

\begin{lemma}\label{lem:rank-D-triangular}
For $s,b\geq1$,
\begin{equation}\label{eq:rank-D-triangular}
D_{(s-1)b}^-(b)
=f_s^{(b)}
+\sum_{1\leq j<s}\lambda_jf_j^{(b)}+V,
\qquad
\lambda_j\in R,
\end{equation}
where $V$ is an $R$-linear combination of ordered products $\prod_{j\geq1}f_j^{(a_j)}$ satisfying
\[
\sum_{j\geq1}a_j=b,
\qquad
\#\{j:a_j>0\}\geq2.
\]
In particular, every nonzero exponent occurring in $V$ is strictly
less than $b$.
\end{lemma}

\begin{proof}
The negative relation in \eqref{eq:prelim-xx} implies
\[
[f_r,f_t]
=\sum_{q=r}^{t-1}f_qf_{r+t-1-q},
\qquad 0\leq r<t.
\]
Thus an adjacent pair in the opposite order may be rewritten as
\begin{equation}\label{eq:reorderingff}
f_tf_r
=f_rf_t-\sum_{q=r}^{t-1}f_qf_{r+t-1-q},
\qquad r<t.
\end{equation}
All terms have the same root degree. The first term has loop degree $r+t$, while every term in the sum
has loop degree $r+t-1$. By the definition of $F(u)$,
\begin{equation}\label{eq:rank-D-word-expansion}
D_{(s-1)b}^-(b)
=\frac{1}{b!}
\sum_{\substack{i_1,\ldots,i_b\geq1\\
i_1+\cdots+i_b=sb}}
f_{i_1}\cdots f_{i_b}.
\end{equation}
Every word in this sum has root degree $b$ and loop degree $sb$.
The preceding commutation procedure therefore produces only ordered
products involving indices $j\geq1$, of root degree $b$ and loop
degree at most $sb$.

The term $f_s^{(b)}$ has loop degree $sb$ and hence cannot arise
from any correction term in \eqref{eq:reorderingff}. In \eqref{eq:rank-D-word-expansion}, it
can arise without correction only from the tuple
$i_1=\cdots=i_b=s$, whose contribution is precisely $f_s^{(b)}$.
Thus its coefficient is $1$.

An ordered product with only one nontrivial divided power factor is
$f_j^{(b)}$. Its loop degree is $jb\leq sb$, so $j\leq s$.
The case $j=s$ has just been determined, and the remaining such
terms have $1\leq j<s$. Every other ordered product has at least
two nonzero exponents whose sum is $b$, so each of these exponents
is strictly less than $b$. Finally, $D_{(s-1)b}^-(b)\in\mathcal Y_R^{<}$ by
\eqref{eq:current-integrality}, and the ordered divided power
products form an $R$-basis of $\mathcal Y_R^{<}$ by
Theorem~\ref{thm:prelim-one-sided-pbw}. Hence all coefficients in
the resulting expansion belong to $R$. This proves
\eqref{eq:rank-D-triangular}.
\end{proof}

\begin{proposition}\label{prop:rank-all-mixed}
For all $s,b\in\mathbb N$,
\[
Mf_s^{(b)}\subseteq M.
\]
In particular,
\[
e_r^{(a)}f_s^{(b)}\in M
\qquad(r,s,a,b\in\mathbb N).
\]
\end{proposition}

\begin{proof}
We shall use the standard $\mathfrak{sl}_2$ identity
\begin{equation}\label{eq:rank-zero-mode}
e_0^{(a)}f_0^{(b)}
=\sum_{j=0}^{\min(a,b)}
f_0^{(b-j)}
\binom{h_0-a-b+2j}{j}
e_0^{(a-j)}.
\end{equation}
In particular, $e_0^{(a)}f_0^{(b)}\in M$. We induct on $b$, simultaneously for all $s$. The case $b=0$ is
immediate. Fix $b\geq1$ and assume
\begin{equation}\label{eq:rank-smaller-right-stability}
Mf_j^{(c)}\subseteq M
\qquad(j\in\mathbb N,\ 0\leq c<b).
\end{equation}
We first prove $e_0^{(a)}f_s^{(b)}\in M$, then
$e_r^{(a)}f_s^{(b)}\in M$, and finally
$Mf_s^{(b)}\subseteq M$.

Fix $a\in\mathbb N$. We prove $e_0^{(a)}f_s^{(b)}\in M$ by
induction on $s$. The cases $s=0$ and $s=1$ follow from
\eqref{eq:rank-zero-mode} and Theorem~\ref{thm:rank-basic},
respectively. Let $s\geq2$ and suppose that
\[
e_0^{(a)}f_j^{(b)}\in M
\qquad(0\leq j<s).
\]
Multiplying \eqref{eq:rank-D-triangular} on the left by
$e_0^{(a)}$ gives
\[
e_0^{(a)}f_s^{(b)}
=e_0^{(a)}D_{(s-1)b}^-(b)
-\sum_{1\leq j<s}\lambda_je_0^{(a)}f_j^{(b)}
-e_0^{(a)}V.
\]
The first term belongs to $M$ by
Lemma~\ref{lem:rank-e0-current}, and the terms in the sum belong to
$M$ by the induction on $s$. Each term in $V$ is a product of
negative divided powers with exponents strictly less than $b$.
Starting from $e_0^{(a)}\in M$ and applying
\eqref{eq:rank-smaller-right-stability} successively shows that
$e_0^{(a)}V\in M$. Hence
\[
e_0^{(a)}f_s^{(b)}\in M
\qquad(a,s\in\mathbb N).
\]

We next induct on $r$, simultaneously for all $a$ and $s$. Suppose
that
\[
e_r^{(a)}f_s^{(b)}\in M
\qquad(a,s\in\mathbb N).
\]
Set $\Theta_z=\Phi_{0,z}$. Since
$\Theta_z^{-1}(f_s)=f_s-zf_{s+1}$, the formula \eqref{eq:prelim-adjacent-binomial} 
gives
\[
\Theta_z^{-1}(f_s^{(b)}) 
=\sum_{j=0}^{b}
(-z)^{b-j}
\left(\prod_{q=0}^{j-1}(1+qz)\right)
f_s^{(j)}f_{s+1}^{(b-j)}.
\]
For $j=0$ or $j=b$, the corresponding product after left
multiplication by $e_r^{(a)}$ belongs to $M$ by the induction on
$r$. If $0<j<b$, both exponents are smaller than $b$, so
\eqref{eq:rank-smaller-right-stability} shows that the corresponding
product also belongs to $M$. Thus
\[
e_r^{(a)}\Theta_z^{-1}(f_s^{(b)})\in M[z].
\]
By \eqref{eq:rank-Phi-M}, applying $\Theta_z$ gives
\[
(e_r-ze_{r+1})^{(a)}f_s^{(b)}\in M[[z]].
\]
The coefficient of $z^a$ in $(e_r-ze_{r+1})^{(a)}$ is
$(-1)^ae_{r+1}^{(a)}$, hence $e_{r+1}^{(a)}f_s^{(b)}\in M$, which completes the induction on $r$.

It remains to prove $Mf_s^{(b)}\subseteq M$. We first record the
root degree estimate needed below. Give $e_j$ and $f_j$ root degree
$1$ and $h_j$ root degree $0$. Since
$[e_r,f_s]=h_{r+s}$, one has
\[
\deg_{\mathrm{rt}}[e_r^{(a)},f_s^{(b)}]
\leq a+b-2.
\]
Moreover, every term in the expansion of this commutator with
respect to the basis \eqref{eq:rank-basis} has $h_0$-weight
$2(a-b)$. Thus, if such a term is written as $u^-hu^+$ and
\[
p=\deg_{\mathrm{rt}}u^- ,
\qquad
q=\deg_{\mathrm{rt}}u^+,
\]
then $q-p=a-b$, $p+q\leq a+b-2$. Consequently, $p\leq b-1$, $q\leq a-1$.

By \eqref{eq:rank-M-stability}, it is enough to prove
$x^+f_s^{(b)}\in M$ for every ordered product $x^+$ of positive
divided powers. We induct on the number of nontrivial factors in
$x^+$. The case $x^+=1$ is immediate. Write
$x^+=y^+e_r^{(a)}$ with $a>0$. Then
\[
x^+f_s^{(b)}
=(y^+f_s^{(b)})e_r^{(a)}
+y^+[e_r^{(a)},f_s^{(b)}].
\]
The first term belongs to $M$ by the induction on the number of
factors and \eqref{eq:rank-M-stability}. Since both
$e_r^{(a)}f_s^{(b)}$ and $f_s^{(b)}e_r^{(a)}$ belong to $M$, the
commutator belongs to $M$. Write its expansion in the basis
\eqref{eq:rank-basis} as
\[
[e_r^{(a)},f_s^{(b)}]
=\sum_\nu u_\nu^-h_\nu u_\nu^+,
\qquad
h_\nu\in\mathcal H_R.
\]
The root degree estimate above gives
$\deg_{\mathrm{rt}}u_\nu^-<b$. Hence every nontrivial divided power
factor occurring in $u_\nu^-$ has exponent less than $b$.
Starting from $y^+\in M$, the induction hypothesis
\eqref{eq:rank-smaller-right-stability} allows successive right
multiplication by these factors. We obtain
$y^+u_\nu^-\in M$, and \eqref{eq:rank-M-stability} then gives
\[
y^+u_\nu^-h_\nu u_\nu^+\in M.
\]
Thus $y^+[e_r^{(a)},f_s^{(b)}]\in M$, and therefore
$x^+f_s^{(b)}\in M$. This proves $Mf_s^{(b)}\subseteq M$ and
completes the induction on $b$.
\end{proof}

\subsection{The PBW theorem}
\label{subsec:rank-pbw}

It remains to show that the Cartan generators introduced in
Section~\ref{subsec:prelim-cartan} belong to $\mathcal Y_R$.
The zero index identity \eqref{eq:rank-zero-mode} yields the
binomial coefficients in $h_0$, while the basic lemma yields the
$P_n$.

\begin{theorem}\label{thm:rank-pbw}
Multiplication induces an isomorphism of $R$-modules
\begin{equation}\label{eq:rank-integral-triangular}
\mathcal Y_R^{<}\otimes_R\mathcal H_R
\otimes_R\mathcal Y_R^{>}
\xrightarrow{\;\sim\;}\mathcal Y_R.
\end{equation}
In particular, the products \eqref{eq:rank-basis} form an $R$-basis
of $\mathcal Y_R$.
\end{theorem}

\begin{proof}
Proposition~\ref{prop:rank-all-mixed} and
\eqref{eq:rank-M-stability} show that right multiplication by every
generator of $\mathcal Y_R$ preserves $M$. Since $1\in M$, every
product of these generators belongs to $M$, and hence $\mathcal Y_R\subseteq M$. For the reverse inclusion, it remains to prove
$\mathcal H_R\subseteq\mathcal Y_R$. We first show that
\[
\binom{h_0}{n}\in\mathcal Y_R
\qquad(n\in\mathbb N)
\]
by induction on $n$. The case $n=0$ is immediate. Suppose $n>0$
and the assertion holds for all smaller indices. Setting $a=b=n$
in \eqref{eq:rank-zero-mode} gives
\[
e_0^{(n)}f_0^{(n)}
=\sum_{j=0}^{n}
f_0^{(n-j)}
\binom{h_0-2n+2j}{j}
e_0^{(n-j)}.
\]
The term with $j=n$ is $\binom{h_0}{n}$. If $j<n$, then
\eqref{eq:prelim-binomial-translation} expresses
$\binom{h_0-2n+2j}{j}$ as an integral linear combination of
$\binom{h_0}{l}$ with $0\leq l\leq j<n$. By induction, all other
terms lie in $\mathcal Y_R$. Since the left hand side also belongs
to $\mathcal Y_R$, it follows that
$\binom{h_0}{n}\in\mathcal Y_R$. Hence $\Bin_R(h_0)\subseteq\mathcal Y_R$.

We next prove $P_n\in\mathcal Y_R$ by induction on $n$. The case
$P_0=1$ is clear. Let $n\geq1$ and suppose that
$P_j\in\mathcal Y_R$ for $j<n$. By \eqref{eq:inverseBx},
\[
[u^{-m}]\frac{Q(u)}{B_{h_0+c}(u)}
=\sum_{j=0}^{m}P_{m-j}\gamma_j(h_0+c)
\qquad(c\in\mathbb Z,\ m\geq0).
\]
Since $\gamma_j(h_0+c)\in\Bin_R(h_0)$, the displayed coefficient
belongs to $\mathcal Y_R$ whenever $m<n$. Apply \eqref{eq:rank-basic} with $a=b=n$. For $t<n$, the $t$-th
summand is
\[
(-1)^t[u^{-t}]
\left(
F(u)^{(n-t)}
\frac{Q(u)}{B_{h_0-2n+t+1}(u)}
E(u)^{(n-t)}
\right).
\]
Extracting the coefficient of $u^{-t}$ uses only Cartan
coefficients of index at most $t<n$, which belong to
$\mathcal Y_R$ by the preceding arguments. The coefficients of the
other two factors belong to $\mathcal Y_R^{<}$ and
$\mathcal Y_R^{>}$ by \eqref{eq:current-integrality}. Thus every
summand with $t<n$ belongs to $\mathcal Y_R$. For $t=n$, the summand is
\begin{equation}\label{eq:QPtriangular}
(-1)^n[u^{-n}]
\frac{Q(u)}{B_{h_0-n+1}(u)}
=(-1)^n
\left(
P_n+\sum_{j=1}^{n}
P_{n-j}\gamma_j(h_0-n+1)
\right).
\end{equation}
Every term in the sum belongs to $\mathcal Y_R$ by induction. The
left hand side of \eqref{eq:rank-basic} also belongs to
$\mathcal Y_R$, and all terms with $t<n$ have just been shown to
belong to $\mathcal Y_R$. Hence $P_n\in\mathcal Y_R$.

We have proved $\mathcal H_R\subseteq\mathcal Y_R$. Therefore
\[
M=\mathcal Y_R^{<}\mathcal H_R\mathcal Y_R^{>}
\subseteq\mathcal Y_R,
\]
and hence $M=\mathcal Y_R$. The products
\eqref{eq:rank-basis} form an $R$-basis of $M$, so they form an
$R$-basis of $\mathcal Y_R$. They are precisely the images under
multiplication of the tensor products of the chosen bases of
$\mathcal Y_R^{<}$, $\mathcal H_R$, and $\mathcal Y_R^{>}$. This
proves \eqref{eq:rank-integral-triangular}.
\end{proof}

\section{Integral forms in classical type}
\label{sec:classical}

\subsection{The Cartan part and the PBW theorem}
\label{subsec:general-cartan}

We return to $Y(\mathfrak g)$, where $\mathfrak g$ is a classical
simple Lie algebra, with the notation fixed in
Subsection~\ref{subsec:yangian-pbw}. For each $i\in I$, the
assignments
\begin{equation}\label{eq:general-rank-one-embedding}
\iota_i(e_r)=d_i^{-r}e_{i,r},
\qquad
\iota_i(f_r)=d_i^{-r}f_{i,r},
\qquad
\iota_i(h_r)=d_i^{-r}h_{i,r}
\end{equation}
define an embedding
\[
\iota_i\colon
Y(\mathfrak{sl}_2)\longrightarrow Y(\mathfrak g).
\]
Since $d_i\in R^\times$, the image of
$\mathcal Y_R(\mathfrak{sl}_2)$ under $\iota_i$ is precisely the
$R$-subalgebra generated by the divided powers
$e_{i,r}^{(a)}$ and $f_{i,r}^{(a)}$. For $i\in I$, put $H_i(u)
=
1+d_i\sum_{r\geq0}h_{i,r}u^{-r-1}$,  and define
\[
Q_i(u)
=
\iota_i\bigl(Q(u/d_i)\bigr)
=
\sum_{m\geq0}P_{i,m}u^{-m}.
\]
Thus $P_{i,m}=d_i^m\iota_i(P_m)$. 
Applying $\iota_i$ to \eqref{eq:rank-difference}, with the variable
replaced by $u/d_i$, gives
\begin{equation}\label{eq:general-difference}
(1+d_i/u)^{h_{i,0}}Q_i(u+d_i)
=
H_i(u)Q_i(u).
\end{equation}
In particular, comparison of the coefficients of $u^{-m-1}$ gives
\begin{equation}\label{eq:general-P-recursion}
mP_{i,m}
=
\sum_{j=0}^{m-1}
\left(
d_i^{m-j}\binom{h_{i,0}-j}{m+1-j}
-h_{i,m-j}
\right)P_{i,j},
\qquad m\geq1.
\end{equation}
It follows inductively that
\begin{equation}\label{eq:general-P-triangular}
P_{i,m}
=
-\frac{h_{i,m}}{m}
+p_{i,m}(h_{i,0},\ldots,h_{i,m-1}),
\qquad
p_{i,m}\in
\mathbb Q[h_{i,0},\ldots,h_{i,m-1}].
\end{equation}
By Theorem~\ref{thm:rank-pbw} and
\eqref{eq:general-rank-one-embedding},
\[
\binom{h_{i,0}}{b},\ P_{i,m}
\in\mathcal Y_R(\mathfrak g)
\qquad
(i\in I,\ b,m\in\mathbb N).
\]
Define
\[
\mathcal H_R(\mathfrak g)
=
R\left[
\binom{h_{i,0}}{b},P_{i,m}
\ \middle|\
i\in I,\ b\in\mathbb N,\ m\geq1
\right]
\subseteq Y^0(\mathfrak g).
\]
Then $\mathcal H_R(\mathfrak g)
\subseteq\mathcal Y_R(\mathfrak g)$.

\begin{lemma}\label{lem:general-Cartan-basis}
The elements
\begin{equation}\label{eq:general-Cartan-basis}
\left(
\prod_{i\in I}\binom{h_{i,0}}{b_i}
\right)
\left(
\prod_{\substack{i\in I\\m\geq1}}
P_{i,m}^{c_{i,m}}
\right),
\end{equation}
where all exponents belong to $\mathbb N$ and only finitely many
$c_{i,m}$ are nonzero, form an $R$-basis of
$\mathcal H_R(\mathfrak g)$.
\end{lemma}

\begin{proof}
The elements in \eqref{eq:general-Cartan-basis} span $\mathcal H_R(\mathfrak g)$  by its definition.
The PBW theorem for $Y(\mathfrak g)$ gives
\[
Y^0(\mathfrak g)
=
\mathbb C[h_{i,r}\mid i\in I,\ r\in\mathbb N].
\]
For each $i$, the triangular relation
\eqref{eq:general-P-triangular} gives
\[
\mathbb C[h_{i,0},h_{i,1},h_{i,2},\ldots]
=
\mathbb C[h_{i,0},P_{i,1},P_{i,2},\ldots].
\]
Hence the elements
\[
h_{i,0},\quad P_{i,m}
\qquad(i\in I,\ m\geq1)
\]
are algebraically independent. Since
\[
\left\{
\prod_{i\in I}\binom{h_{i,0}}{b_i}
\ \middle|\
b_i\in\mathbb N
\right\}
\]
is an $R$-basis of the product of the binomial algebras in the
variables $h_{i,0}$, the elements
\eqref{eq:general-Cartan-basis} are $R$-linearly independent.
\end{proof}

For a fixed $j\in I$, let $S$ act on the generators of the
$j$-th simple root subalgebra by
\[
Se_{j,r}=e_{j,r+1},
\qquad
Sf_{j,r}=f_{j,r+1}.
\]
The Cartan--root relations give
\begin{equation}\label{eq:general-H-conjugation}
H_i(u)x_{j,r}^{\pm}H_i(u)^{-1}
=
\frac{u-S\pm d_i a_{ij}/2}
     {u-S\mp d_i a_{ij}/2}
x_{j,r}^{\pm}.
\end{equation}

\begin{lemma}\label{lem:general-Q-conjugation}
For $i,j\in I$ and $r\geq0$,
\begin{equation}\label{eq:general-Q-conjugation}
\begin{aligned}
Q_i(u)e_{j,r}Q_i(u)^{-1}
&=
K_{ij}(u,S)e_{j,r},\\
Q_i(u)^{-1}f_{j,r}Q_i(u)
&=
K_{ij}(u,S)f_{j,r},
\end{aligned}
\end{equation}
where
\begin{equation}\label{eq:general-K}
K_{ij}(u,S)
=
\begin{cases}
\displaystyle
\frac{(u-S)(u-S-d_i)}{u^2},
&i=j,\\[8pt]
1,
&a_{ij}=0,\\[6pt]
\displaystyle
\frac{u^k}
{\prod_{\ell=0}^{k-1}
\bigl(u-S+d_i(\ell-k/2)\bigr)},
&a_{ij}=-k<0.
\end{cases}
\end{equation}
\end{lemma}

\begin{proof}
Put $G_i(u)
=
(1+d_i/u)^{-h_{i,0}}H_i(u)
=
1+O(u^{-2})$. By \eqref{eq:general-difference}, $Q_i(u+d_i)=G_i(u)Q_i(u)$, and thus $X(u)=Q_i(u)e_{j,r}Q_i(u)^{-1}$ satisfies
\[
X(u+d_i)
=
G_i(u)X(u)G_i(u)^{-1},
\qquad
X(u)=e_{j,r}+O(u^{-1}).
\]
As in the proof of Lemma~\ref{lem:rank-Q-conjugation}, this
difference equation has at most one solution with the prescribed
constant term.  Since $[h_{i,0},e_{j,r}]=a_{ij}e_{j,r}$,  \eqref{eq:general-H-conjugation} gives
\[
G_i(u)e_{j,r}G_i(u)^{-1}
=
(1+d_i/u)^{-a_{ij}}
\frac{u-S+d_i a_{ij}/2}
     {u-S-d_i a_{ij}/2}
e_{j,r}.
\]
A direct calculation from \eqref{eq:general-K} shows that
\[
\frac{K_{ij}(u+d_i,S)}{K_{ij}(u,S)}
=
(1+d_i/u)^{-a_{ij}}
\frac{u-S+d_i a_{ij}/2}
     {u-S-d_i a_{ij}/2}.
\]
Thus $K_{ij}(u,S)e_{j,r}$ satisfies the same difference equation
and has constant term $e_{j,r}$. The uniqueness then proves the first
identity in \eqref{eq:general-Q-conjugation}. By applying the anti-automorphism $\tau$, we obtain the second identity in \eqref{eq:general-Q-conjugation}.
\end{proof}

By the same argument as in the proof of  Proposition \ref{prop:rank-cartan-exchange}, we obtain
\begin{proposition}\label{prop:general-Cartan-exchange}
One has
\begin{equation}\label{eq:general-integral-exchange}
\begin{aligned}
\mathcal H_R(\mathfrak g)\mathcal Y_R^{>}(\mathfrak g)
&=
\mathcal Y_R^{>}(\mathfrak g)\mathcal H_R(\mathfrak g),\\
\mathcal H_R(\mathfrak g)\mathcal Y_R^{<}(\mathfrak g)
&=
\mathcal Y_R^{<}(\mathfrak g)\mathcal H_R(\mathfrak g).
\end{aligned}
\end{equation}
\end{proposition}
Let $M(\mathfrak g)
=
\mathcal Y_R^{<}(\mathfrak g)
\mathcal H_R(\mathfrak g)
\mathcal Y_R^{>}(\mathfrak g)$. By
\eqref{eq:general-integral-exchange},
\begin{equation}\label{eq:general-M-stability}
\mathcal Y_R^{<}(\mathfrak g)M(\mathfrak g)
\subseteq M(\mathfrak g),
\qquad
\mathcal H_R(\mathfrak g)M(\mathfrak g)
\subseteq M(\mathfrak g),
\end{equation}
and
\[
M(\mathfrak g)\mathcal H_R(\mathfrak g)
\subseteq M(\mathfrak g),
\qquad
M(\mathfrak g)\mathcal Y_R^{>}(\mathfrak g)
\subseteq M(\mathfrak g).
\]

\begin{lemma}\label{lem:general-right-stability}
For every $j\in I$ and $s,b\in\mathbb N$,
\[
M(\mathfrak g)f_{j,s}^{(b)}
\subseteq M(\mathfrak g).
\]
\end{lemma}

\begin{proof}
We induct on $b$, simultaneously for all $j$ and $s$. The case
$b=0$ is immediate. Suppose $b>0$ and
\[
M(\mathfrak g)f_{j,s}^{(c)}
\subseteq M(\mathfrak g)
\qquad
(j\in I,\ s\in\mathbb N,\ 0\leq c<b).
\]
By \eqref{eq:general-M-stability}, it is enough to prove that $x^+f_{j,s}^{(b)}\in M(\mathfrak g)$ for every word $x^+$ in the simple root positive divided powers.
We induct on the number of factors in such a word. The assertion is clear for $x^+=1$. Write $x^+=y^+e_{i,r}^{(a)}$. If $i\ne j$, then
\[
e_{i,r}^{(a)}f_{j,s}^{(b)}
=
f_{j,s}^{(b)}e_{i,r}^{(a)},
\]
so the induction on the number of factors gives
\[
x^+f_{j,s}^{(b)}
=
(y^+f_{j,s}^{(b)})e_{i,r}^{(a)}
\in M(\mathfrak g).
\]
If $i=j$, applying \eqref{eq:general-rank-one-embedding} to
the rank one calculation in the proof of
Proposition~\ref{prop:rank-all-mixed} gives
\[
e_{i,r}^{(a)}f_{i,s}^{(b)}
=
f_{i,s}^{(b)}e_{i,r}^{(a)}
+
\sum_\nu
u_\nu^-h_\nu u_\nu^+,
\]
where $h_\nu\in\mathcal H_R(\mathfrak g)$, $u_\nu^+\in\mathcal Y_R^{>}(\mathfrak g)$, and every divided power factor occurring in $u_\nu^-$ has exponent
strictly less than $b$. The first term gives
\[
y^+f_{i,s}^{(b)}e_{i,r}^{(a)}
\in M(\mathfrak g)
\]
by the induction on the number of positive factors. For each term
in the sum, the induction on $b$ allows us to multiply $y^+$
successively on the right by the divided power factors of
$u_\nu^-$. Thus $y^+u_\nu^-\in M(\mathfrak g)$, and the remaining Cartan and positive factors preserve
$M(\mathfrak g)$. Hence $y^+u_\nu^-h_\nu u_\nu^+
\in M(\mathfrak g)$. This completes both inductions.
\end{proof}

\begin{theorem}\label{thm:general-pbw}
Multiplication induces an isomorphism of $R$-modules
\[
\mathcal Y_R^{<}(\mathfrak g)
\otimes_R\mathcal H_R(\mathfrak g)
\otimes_R\mathcal Y_R^{>}(\mathfrak g)
\xrightarrow{\;\sim\;}
\mathcal Y_R(\mathfrak g).
\]
More explicitly, the products
\begin{equation}\label{eq:general-PBW-basis}
\left(
\prod_{\beta,r}^{\leftarrow}
f_\beta(r)^{(a_{\beta,r})}
\right)
\left(
\prod_{i\in I}\binom{h_{i,0}}{b_i}
\right)
\left(
\prod_{\substack{i\in I\\m\geq1}}
P_{i,m}^{c_{i,m}}
\right)
\left(
\prod_{\beta,r}^{\rightarrow}
e_\beta(r)^{(d_{\beta,r})}
\right)
\end{equation}
form an $R$-basis of $\mathcal Y_R(\mathfrak g)$, where the order is
the one fixed in \eqref{eq:prelim-root-order} and all exponent
families have finite support.
\end{theorem}

\begin{proof}
By Theorem~\ref{thm:prelim-one-sided-pbw},
Lemma~\ref{lem:general-Cartan-basis}, and the triangular
decomposition \eqref{eq:prelim-rational-triangular}, the products
\eqref{eq:general-PBW-basis} form an $R$-basis of
$M(\mathfrak g)$. Since $\mathcal H_R(\mathfrak g)
\subseteq\mathcal Y_R(\mathfrak g)$, we have $M(\mathfrak g)\subseteq\mathcal Y_R(\mathfrak g)$. Conversely, Lemma~\ref{lem:general-right-stability} and the
inclusions preceding it show that right multiplication by every
generator $e_{i,r}^{(a)}$, $f_{i,r}^{(a)}$ of $\mathcal Y_R(\mathfrak g)$ preserves $M(\mathfrak g)$.
Since $1\in M(\mathfrak g)$, every product of these generators
belongs to $M(\mathfrak g)$. Hence $\mathcal Y_R(\mathfrak g)\subseteq M(\mathfrak g)$. Thus $\mathcal Y_R(\mathfrak g)=M(\mathfrak g)$, and the asserted basis and tensor product decomposition follow.
\end{proof}

\subsection{Relation with Garland's integral form}
\label{subsec:associated-graded}

Equip $Y(\mathfrak g)$ with the loop filtration
\[
\deg e_{i,r}
=
\deg f_{i,r}
=
\deg h_{i,r}
=
r.
\]
For $x\in F_rY(\mathfrak g)$, write
$\operatorname{gr}_r x$ for its image in
$F_rY(\mathfrak g)/F_{r-1}Y(\mathfrak g)$. The PBW theorem gives
an isomorphism
\[
\operatorname{gr}Y(\mathfrak g)
\simeq U(\mathfrak g[t])
\]
under which
\[
\operatorname{gr}_r e_{i,r}
\longmapsto e_i\otimes t^r,
\qquad
\operatorname{gr}_r f_{i,r}
\longmapsto f_i\otimes t^r,
\qquad
\operatorname{gr}_r h_{i,r}
\longmapsto h_i\otimes t^r.
\]

Garland constructed an integral form of the enveloping algebra of
the derived untwisted affine Lie algebra and proved an integral PBW
theorem in \cite[Theorem~5.8]{Gar}. The quotient of that Lie algebra
by its centre is $L\mathfrak g=\mathfrak g[t,t^{-1}]$, see \cite[Theorem~2.1]{Gar}. Let
$U_{\mathbb Z}(L\mathfrak g)$ be the image of Garland's integral
form in $U(L\mathfrak g)$, and set
\[
U_{\mathbb Z}(\mathfrak g[t])
=
U_{\mathbb Z}(L\mathfrak g)
\cap U(\mathfrak g[t]),
\qquad
U_R(\mathfrak g[t])
=
R\otimes_{\mathbb Z}
U_{\mathbb Z}(\mathfrak g[t]).
\]

For $\beta\in\Delta^+$, let $e_\beta\in\mathfrak g_\beta$ and
$f_\beta\in\mathfrak g_{-\beta}$ be the root vectors defined by the
same commutator expressions as the Lyndon root vectors
$e_\beta(0)$ and $f_\beta(0)$. Put
\[
\mathsf e_\beta(r)=e_\beta\otimes t^r,
\qquad
\mathsf f_\beta(r)=f_\beta\otimes t^r,
\qquad
\mathsf h_i(r)=h_i\otimes t^r.
\]
Define the Garland Cartan elements $\Lambda_{i,m}$ by
\begin{equation}\label{eq:Garland-Cartan-series}
\sum_{m\geq0}\Lambda_{i,m}z^m
=
\exp\left(
-\sum_{r\geq1}\frac{\mathsf h_i(r)}{r}z^r
\right),
\qquad
\Lambda_{i,0}=1.
\end{equation}
The current algebra form of Garland's PBW theorem gives, after
extension of scalars to $R$, the basis
\begin{equation}\label{eq:current-PBW-basis}
\left(
\prod_{\beta,r}^{\leftarrow}
\mathsf f_\beta(r)^{(a_{\beta,r})}
\right)
\left(
\prod_{i\in I}\binom{h_i}{b_i}
\right)
\left(
\prod_{\substack{i\in I\\m\geq1}}
\Lambda_{i,m}^{c_{i,m}}
\right)
\left(
\prod_{\beta,r}^{\rightarrow}
\mathsf e_\beta(r)^{(d_{\beta,r})}
\right)
\end{equation}
of $U_R(\mathfrak g[t])$, see also
\cite[Section~1.3]{BMM}.

\begin{proposition}\label{prop:associated-graded}
Give $\mathcal Y_R(\mathfrak g)$ the filtration induced from
$Y(\mathfrak g)$. Under the identification
$\operatorname{gr}Y(\mathfrak g)\simeq U(\mathfrak g[t])$, one has
\[
\operatorname{gr}\mathcal Y_R(\mathfrak g)
=
U_R(\mathfrak g[t]).
\]
Moreover, the classes of the basis elements
\eqref{eq:general-PBW-basis} are precisely the basis elements
\eqref{eq:current-PBW-basis}.
\end{proposition}

\begin{proof}
The recurrence \eqref{eq:general-P-recursion} shows inductively that $P_{i,m}\in F_mY(\mathfrak g)$. Indeed, the term involving
$\binom{h_{i,0}-j}{m+1-j}P_{i,j}$ has degree at most $j<m$, whereas
$h_{i,m-j}P_{i,j}$ has degree at most $m$. Taking the degree-$m$
part of \eqref{eq:general-P-recursion} gives
\[
m\,\operatorname{gr}_mP_{i,m}
=
-\sum_{r=1}^{m}
\mathsf h_i(r)\,
\operatorname{gr}_{m-r}P_{i,m-r}.
\]
On the other hand, differentiating
\eqref{eq:Garland-Cartan-series} and comparing coefficients gives
\[
m\Lambda_{i,m}
=
-\sum_{r=1}^{m}
\mathsf h_i(r)\Lambda_{i,m-r},
\qquad
\Lambda_{i,0}=1.
\]
Since the two recurrences have the same initial term, induction on
$m$ gives
\[
\operatorname{gr}_mP_{i,m}
=
\Lambda_{i,m}.
\]

By the recursive definitions of the Lyndon root vectors and the
normalization fixed in
\cite[Sections~2.1--2.2]{CHH},
\[
\operatorname{gr}_r e_\beta(r)
=
\mathsf e_\beta(r),
\qquad
\operatorname{gr}_r f_\beta(r)
=
\mathsf f_\beta(r).
\]
Consequently, the highest degree class of every element in
\eqref{eq:general-PBW-basis} is the corresponding element of
\eqref{eq:current-PBW-basis}. In particular, if
\[
D
=
\sum_{\beta,r}r(a_{\beta,r}+d_{\beta,r})
+
\sum_{i,m}mc_{i,m},
\]
then the corresponding Yangian basis element has filtration degree
exactly $D$.

The elements \eqref{eq:current-PBW-basis} are linearly
independent. Thus a nonzero $R$-linear combination of
the Yangian basis elements has degree equal to the
largest degree among its nonzero terms.
It follows that $F_N\mathcal Y_R(\mathfrak g)$ has
an $R$-basis consisting of the basis elements of
degree at most $N$. The classes of those of degree
exactly $N$ form an $R$-basis of $F_N\mathcal Y_R(\mathfrak g)/
 F_{N-1}\mathcal Y_R(\mathfrak g)$. Hence
$\operatorname{gr}\mathcal Y_R(\mathfrak g)
=U_R(\mathfrak g[t])$.
\end{proof}

\begin{corollary}\label{cor:graded-base-change}
Let $\kk$ be a field of characteristic $p>3$. With the
induced filtration,
\[
\operatorname{gr}
\bigl(\kk\otimes_R\mathcal Y_R(\mathfrak g)\bigr)
\simeq
\kk\otimes_RU_R(\mathfrak g[t]).
\]
\end{corollary}

\begin{proof}
By the proof of Proposition~\ref{prop:associated-graded}, every
filtered piece of $\mathcal Y_R(\mathfrak g)$ is spanned by a subset
of the basis \eqref{eq:general-PBW-basis}. Thus the filtration splits
as a filtration of free $R$-modules, and taking associated graded
commutes with extension of scalars from $R$ to $\kk$.
\end{proof}

\section{Small and restricted Yangians}
\label{sec:small-yangian}

Throughout the remainder of this  paper, let $\mathbbm{k}$ be an algebraically closed field of characteristic $p>3$. 
The positive small Yangian was identified with the positive
restricted Yangian in \cite[Theorem~5.8]{CHH}. In this section, we extend this
identification to the full Yangian. Put
\begin{equation}\label{eq:hyperyangian}
\mathscr Y_\kk(\mathfrak g)
=
\kk\otimes_R\mathcal Y_R(\mathfrak g).
\end{equation}
We call this algebra the \emph{hyper Yangian}, by analogy
with the hyper loop algebras studied in \cite{JM}. The images of the elements \eqref{eq:general-PBW-basis} form a
$\kk$-basis of $\mathscr Y_\kk(\mathfrak g)$. For
$x\in\mathcal Y_R(\mathfrak g)$, write $\overline x=1\otimes x$. Define the small Yangian by
\[
y_\kk(\mathfrak g)
=
\kk\langle
\overline e_{i,r},\overline f_{i,r}
\mid i\in I,\ r\in\mathbb N
\rangle
\subseteq\mathscr Y_\kk(\mathfrak g).
\]
It also contains $\overline h_{i,r}
=
[\overline e_{i,r},\overline f_{i,0}]$.

Let $Y_\kk(\mathfrak g)$ be the $\kk$-algebra defined by the Drinfeld
relations
\eqref{eq:prelim-hh}--\eqref{eq:prelim-serre}, which is the modular Yangian associated to $\mathfrak g$. Reduction of these
relations gives a surjective homomorphism
\begin{equation}\label{eq:small-reduction-map}
\begin{aligned}
\eta\colon Y_\kk(\mathfrak g)&\longrightarrow y_\kk(\mathfrak g),\\
e_{i,r}&\longmapsto\overline e_{i,r},\\
f_{i,r}&\longmapsto\overline f_{i,r},\\
h_{i,r}&\longmapsto\overline h_{i,r}.
\end{aligned}
\end{equation}
We use the same Lyndon root vectors
$e_\beta(r)$ and $f_\beta(r)$ in $Y_\kk(\mathfrak g)$. Since their
divided powers belong to $\mathcal Y_R(\mathfrak g)$,
\begin{equation}\label{eq:small-root-powers}
\eta(e_\beta(r)^p)
=
\overline{p!\,e_\beta(r)^{(p)}}=0,
\qquad
\eta(f_\beta(r)^p)=0.
\end{equation}
We next determine which Cartan elements belong to the kernel of $\eta$. Recall that $H_i(u)
=
1+d_i\sum_{r\geq0}h_{i,r}u^{-r-1}$, and define
\[
N_i(u)
=
\prod_{a\in\mathbb F_p}H_i(u+a d_i),\qquad 
\zeta_{i,r}
=
d_i^{-p}
[u^{-p(r+1)}] N_i(u),\ \  r\in\mathbb N,
\]
where all shifted series are expanded at infinity. These are
analogues of the Cartan series considered in
\cite[Section~6.2]{BT}.

\begin{proposition}\label{prop:small-Cartan-relations}
Every coefficient of $N_i(u)$ is central in $Y_\kk(\mathfrak g)$.
Moreover,
\begin{equation}\label{eq:small-Cartan-leading}
\zeta_{i,0}
=
h_{i,0}^p-h_{i,0},
\qquad
\zeta_{i,r}
=
h_{i,r}^p-h_{i,pr}+\rho_{i,r}
\quad(r\geq1),
\end{equation}
where $\rho_{i,r}$ is a polynomial in the $h_{i,s}$ in which every
monomial has loop degree strictly less
than $pr$.
\end{proposition}

\begin{proof}
The conjugation formula \eqref{eq:general-H-conjugation} remains
valid after reduction to $\kk$. Hence
\[
N_i(u)x_{j,s}^{\pm}N_i(u)^{-1}
=
\prod_{a\in\mathbb F_p}
\frac{u+d_i(a\pm a_{ij}/2)-S}
     {u+d_i(a\mp a_{ij}/2)-S}
x_{j,s}^{\pm}.
\]
Since translation by $\pm a_{ij}/2$ permutes $\mathbb F_p$, the
factors in the numerator and denominator form the same set. Thus
the product is $1$. Since the coefficients of $N_i(u)$ also commute
with all Cartan generators, they are central.  For \eqref{eq:small-Cartan-leading}, substitute
$u=d_i v$, then
\[
 H_i(d_i v)
 =
 1+\sum_{s\geq0}d_i^{-s}h_{i,s}v^{-s-1}.
\]
We apply \cite[Lemma~2.9]{BT} with $X^{(m)}=d_i^{1-m}h_{i,m-1}$,  and
\[
 \prod_{a\in\mathbb F_p}H_i(d_i(v-a))
 =
 \prod_{a\in\mathbb F_p}H_i(d_i(v+a))
 =
 N_i(d_i v),
\]
Taking the coefficient of $v^{-p(r+1)}$ gives
\[
 [v^{-p(r+1)}]N_i(d_i v)
 =
 d_i^{-pr}(h_{i,r}^p-h_{i,pr})
 +\tilde{\rho}_{i,r},
\]
where $\tilde{\rho}_{i,r}$ is a linear combination
of monomials of loop degree strictly less than $pr$. Hence
\[
 \begin{aligned}
 \zeta_{i,r}=
 d_i^{-p}[u^{-p(r+1)}]N_i(u)=
 h_{i,r}^p-h_{i,pr}+d_i^{pr}\tilde{\rho}_{i,r}.
 \end{aligned}
\]
For $r=0$, again by  \cite[Lemma~2.9]{BT}, we obtain $\zeta_{i,0}=h_{i,0}^p-h_{i,0}$.
\end{proof}

\begin{lemma}\label{lem:small-norm-vanishing}
One has $\eta(N_i(u))=1$. In particular,
\[
\eta(\zeta_{i,r})=0
\qquad(i\in I,\ r\in\mathbb N).
\]
\end{lemma}

\begin{proof}
We first work in the integral form. Multiplying
\eqref{eq:general-difference} at $u, u+d_i, \ldots, u+(p-1)d_i$ gives
\[
\prod_{a=0}^{p-1}H_i(u+a d_i)
=
(1+p d_i/u)^{h_{i,0}}
\frac{Q_i(u+p d_i)}{Q_i(u)}.
\]
All coefficients on the right belong to
$\mathcal H_R(\mathfrak g)$. Reducing coefficientwise modulo $p$,
we have
\[
(1+p d_i/u)^{h_{i,0}}\longmapsto1,\qquad Q_i(u+p d_i)\longmapsto Q_i(u).
\]
Since $Q_i(u)$ has constant coefficient $1$, it is invertible as a
formal series after reduction. The displayed identity therefore
reduces to $\eta(N_i(u))=1$.
\end{proof}

Following \cite{GT}, define the restricted Yangian by
\[
Y_\kk^{[p]}(\mathfrak g)
=
Y_\kk(\mathfrak g)\Big/
\left\langle
e_\beta(r)^p,\,
f_\beta(r)^p,\,
\zeta_{i,r}
\ \middle|\
\beta\in\Delta^+,\ i\in I,\ r\in\mathbb N
\right\rangle.
\]
In types $A,B,D$, the positive root powers are central by
\cite[Theorem~4.6]{CHH}, and the negative ones are central by the
anti-automorphism $\tau$. In type $C$, their centrality is
conjectured in \cite[Remark~4.7]{CHH}. We shall need the following consequence of Garland's Cartan basis.

\begin{lemma}\label{lem:small-Cartan-independence}
In $\kk\otimes_R U_R(\mathfrak g[t])$, the elements
\[
 \prod_{i\in I,r\geq0}
 \mathsf h_i(r)^{b_{i,r}},
 \qquad 0\leq b_{i,r}<p,
\]
with finite support, are linearly independent.
\end{lemma}

\begin{proof}
By \eqref{eq:current-PBW-basis}, the Cartan part is the tensor
product of the degree zero binomial algebras and $\kk[\Lambda_{i,m}\mid i\in I,\ m\geq1]$. Fix $i$, 
from \eqref{eq:Garland-Cartan-series}, we obtain
\[
 \mathsf h_i(r)
 =
 -r\Lambda_{i,r}
 -\sum_{s=1}^{r-1}
 \mathsf h_i(s)\Lambda_{i,r-s}.
\]
Since the $\Lambda_{i,m}$ are algebraically independent
and $-r$ is invertible when $p\nmid r$,  the elements
\[
 \mathsf h_i(r),\qquad r>0,\quad p\nmid r,
\]
are algebraically independent. To describe their $p$th powers, identify
$\kk[\Lambda_{i,m}\mid m\geq1]$ with the ring of
symmetric functions over $\kk$ by
\[
 \Lambda_{i,m}\longmapsto(-1)^m e_m,
\]
where $e_m$ is the elementary symmetric function.
The preceding recurrence is Newton's identity,
so under this identification $\mathsf h_i(r)$
corresponds to the power sum
$p_r=\sum_\nu x_\nu^r$. In characteristic $p$,
\[
 p_r^p
 =
 \left(\sum_\nu x_\nu^r\right)^p
 =
 \sum_\nu x_\nu^{pr}
 =
 p_{pr}.
\]
Consequently,
\begin{equation}\label{eq:small-current-Cartan}
 \mathsf h_i(pr)=\mathsf h_i(r)^p
 \qquad(r\geq1).
\end{equation}

Every positive integer is uniquely of the form
$p^q n$ with $p\nmid n$. Repeated application of
\eqref{eq:small-current-Cartan} therefore gives
\[
 \prod_{r\geq1}\mathsf h_i(r)^{b_{i,r}}
 =
 \prod_{n\geq1, p\nmid n}
 \mathsf h_i(n)^{
   \sum_{q\geq0}p^q b_{i,p^q n}}.
\]
Since $0\leq b_{i,r}<p$, uniqueness of base-$p$
expansions shows that distinct exponent families
give distinct monomials on the right. Hence the products involving $r>0$ are
linearly independent. For $r=0$, the elements
\[
1,h_i,\ldots,h_i^{p-1}
\]
are linearly independent in the binomial algebra, since their
expressions in
\[
\binom{h_i}{0},\ldots,\binom{h_i}{p-1}
\]
are triangular with diagonal entries
$0!,1!,\ldots,(p-1)!$, all nonzero in $\kk$.
\end{proof}

Finally, we have the following theorem, which extends \cite[Theorem~5.8]{CHH} to the full Yangian:

\begin{theorem}\label{thm:small-restricted}
The homomorphism $\eta$ of  \eqref{eq:small-reduction-map} induces an
isomorphism
\[
Y_\kk^{[p]}(\mathfrak g)
\xrightarrow{\;\sim\;}
y_\kk(\mathfrak g).
\]
The images of the products
\begin{equation}\label{eq:small-PBW-basis}
\left(
\prod_{\beta,r}^{\leftarrow}
f_\beta(r)^{a_{\beta,r}}
\right)
\left(
\prod_{i\in I, r\geq0}
h_{i,r}^{b_{i,r}}
\right)
\left(
\prod_{\beta,r}^{\rightarrow}
e_\beta(r)^{c_{\beta,r}}
\right),
\qquad
0\leq
a_{\beta,r},b_{i,r},c_{\beta,r}<p,
\end{equation}
with finite support, form a $\kk$-basis of both algebras.
\end{theorem}

\begin{proof}
By the defining 
relations
\eqref{eq:prelim-hh}--\eqref{eq:prelim-serre}, $Y_\kk(\mathfrak g)$ is spanned by ordered products
of negative root vectors, Cartan generators, and positive root
vectors. Therefore \eqref{eq:small-Cartan-leading} gives that the
products \eqref{eq:small-PBW-basis} span
$Y_\kk^{[p]}(\mathfrak g)$. By \eqref{eq:small-root-powers} and
Lemma~\ref{lem:small-norm-vanishing}, the map $\eta$ induces a
surjection
\[
Y_\kk^{[p]}(\mathfrak g)
\longrightarrow
y_\kk(\mathfrak g).
\]
By Corollary~\ref{cor:graded-base-change}, $\operatorname{gr}\mathscr Y_\kk(\mathfrak g)
\simeq
\kk\otimes_RU_R(\mathfrak g[t])$. Under this identification,
\[
\operatorname{gr}_r\eta(e_\beta(r))
=
\mathsf e_\beta(r),
\qquad
\operatorname{gr}_r\eta(f_\beta(r))
=
\mathsf f_\beta(r),
\qquad
\operatorname{gr}_r\eta(h_{i,r})
=
\mathsf h_i(r).
\]
Lemma~\ref{lem:small-Cartan-independence} gives the linear
independence of the Cartan factors with exponents less than $p$. Together with \eqref{eq:current-PBW-basis}, this
shows that the products
\[
 \left(
 \prod_{\beta,r}^{\leftarrow}
 \mathsf f_\beta(r)^{a_{\beta,r}}
 \right)
 \left(
 \prod_{\substack{i\in I\\r\geq0}}
 \mathsf h_i(r)^{b_{i,r}}
 \right)
 \left(
 \prod_{\beta,r}^{\rightarrow}
 \mathsf e_\beta(r)^{c_{\beta,r}}
 \right),
 \qquad
 0\leq a_{\beta,r},b_{i,r},c_{\beta,r}<p,
\]
are linearly independent in
$\kk\otimes_R U_R(\mathfrak g[t])$.
Under the identification
\[
 \operatorname{gr}\mathscr Y_\kk(\mathfrak g)
 \simeq \kk\otimes_R U_R(\mathfrak g[t]),
\]
these are the classes of the images under $\eta$
of the products \eqref{eq:small-PBW-basis}. Thus the images are linearly independent in
$y_\kk(\mathfrak g)$, this completes the proof.
\end{proof}

Let $y_\kk^{<}(\mathfrak g)$, $y_\kk^0(\mathfrak g)$, and
$y_\kk^{>}(\mathfrak g)$ be the subalgebras generated by the $\overline{f}_{i,r}$, $\overline{h}_{i,r}$, and $\overline{e}_{i,r}$, respectively.

\begin{corollary}\label{cor:small-triangular}
Multiplication induces an isomorphism of $k$-vector spaces
\[
y_\kk^{<}(\mathfrak g)
\otimes_k
y_\kk^0(\mathfrak g)
\otimes_k
y_\kk^{>}(\mathfrak g)
\xrightarrow{\;\sim\;}
y_\kk(\mathfrak g).
\]
\end{corollary}

\begin{proof}
This follows immediately from the basis
\eqref{eq:small-PBW-basis}.
\end{proof}

\section{Reduction modulo $p$ of highest-weight modules}
\label{sec:rep-lattices}

The integral form constructed above allows us to construct
finite dimensional modules for $\mathscr Y_\kk(\mathfrak g)$
from highest weight modules in characteristic zero. We first recall the highest weight theory of Yangians in
characteristic zero and Drinfeld's classification over
$\mathbb C$. We then extend the required results to arbitrary
fields of characteristic zero and construct stable lattices
over a discrete valuation ring with residue field $\kk$.

For a commutative $R$-algebra $S$, set
\[
\mathcal Y_S=S\otimes_R\mathcal Y_R(\mathfrak g),
\qquad
\mathcal H_S=S\otimes_R\mathcal H_R(\mathfrak g),
\]
and use the superscripts ${<}$ and ${>}$ for the corresponding
negative and positive subalgebras.
By Theorem~\ref{thm:general-pbw}, multiplication gives
\[
\mathcal Y_S^{<}\otimes_S\mathcal H_S
\otimes_S\mathcal Y_S^{>}
\xrightarrow{\;\sim\;}\mathcal Y_S,
\]
and the bases over $R$ give bases over $S$ by extension
of scalars. For a field $\mathbb F$ of characteristic zero, write
\[
Y_{\mathbb F}(\mathfrak g)
=\mathcal Y_{\mathbb F}
=\mathbb F\otimes_R\mathcal Y_R(\mathfrak g),
\]
this is
the Yangian over $\mathbb F$ defined by the Drinfeld
presentation in Subsection~\ref{subsec:yangian-pbw}.
In particular, $Y_{\mathbb C}(\mathfrak g)=Y(\mathfrak g)$.
For $S=\kk$, we have
$\mathcal Y_\kk=\mathscr Y_\kk(\mathfrak g)$,
the hyper Yangian defined in \eqref{eq:hyperyangian}.

\subsection{Highest weight modules in characteristic zero}

Let $\mathbb F$ be a field of characteristic zero, not
necessarily algebraically closed. The generators
$e_{i,0}$, $f_{i,0}$, and $h_{i,0}$ generate a
subalgebra of $Y_{\mathbb F}(\mathfrak g)$ isomorphic
to $U_{\mathbb F}(\mathfrak g)$.
Let $\mathfrak h_{\mathbb F}$ be the span of the
$h_{i,0}$, and set $Q^+=\sum_{i\in I}\mathbb N\alpha_i$. The simple roots satisfy
$\alpha_j(h_{i,0})=a_{ij}$.
For a $Y_{\mathbb F}(\mathfrak g)$-module $V$ and
$\lambda\in\mathfrak h_{\mathbb F}^*$, define
\[
V_\lambda
=
\{v\in V\ |\ 
h_{i,0}v=\lambda(h_{i,0})v
\text{ for all }i\in I\}.
\]
These are the $\mathfrak g$-weight spaces of $V$. A highest vector is a nonzero vector $v$ satisfying
\[
e_{i,r}v=0,
\qquad
h_{i,r}v=\mu_{i,r}v
\qquad(i\in I,\ r\geq0)
\]
for some scalars $\mu_{i,r}\in\mathbb F$.
A module generated by such a vector is called a
highest weight module. Its highest weight is recorded
by the tuple $\boldsymbol\psi=(\psi_i(u))_{i\in I}$,
where
\[
\psi_i(u)
=
1+d_i\sum_{r\geq0}\mu_{i,r}u^{-r-1}.
\]
Thus $H_i(u)v=\psi_i(u)v$. For a given tuple $\boldsymbol\psi$,  the corresponding Verma module (for example, see \cite{BFM}) is
\begin{equation}\label{eq:rep-Verma-definition}
M_{\mathbb F}(\boldsymbol\psi)
=
Y_{\mathbb F}(\mathfrak g)\Big/
\left(
\sum_{i,r}Y_{\mathbb F}(\mathfrak g)e_{i,r}
+
\sum_{i,r}Y_{\mathbb F}(\mathfrak g)
(h_{i,r}-\mu_{i,r})
\right).
\end{equation}
Every highest weight module with highest weight
$\boldsymbol\psi$ is a quotient of
$M_{\mathbb F}(\boldsymbol\psi)$.  Let $v_{\boldsymbol\psi}$ be the image of $1$. By the PBW theorem,
$M_{\mathbb F}(\boldsymbol\psi)$ is identified,
as a vector space, with
$\mathcal Y_{\mathbb F}^{<}v_{\boldsymbol\psi}$.
In particular, the vectors
\[
\left(
\prod_{\beta,r}^{\leftarrow}
f_\beta(r)^{a_{\beta,r}}
\right)v_{\boldsymbol\psi},
\qquad a_{\beta,r}\in\mathbb N,
\]
with finite support form a basis. Let
$\lambda_{\boldsymbol\psi}\in\mathfrak h_{\mathbb F}^*$
be defined by
\[
\lambda_{\boldsymbol\psi}(h_{i,0})=\mu_{i,0}
\qquad(i\in I).
\] 
Regarded as an $\mathfrak{g}$-module, we have
\[
M_{\mathbb F}(\boldsymbol\psi)_{\lambda_{\boldsymbol\psi}}
=\mathbb Fv_{\boldsymbol\psi},
\qquad
\operatorname{wt}M_{\mathbb F}(\boldsymbol\psi)
\subseteq\lambda_{\boldsymbol\psi}-Q^+.
\]
Every submodule is a sum of its $\mathfrak g$-weight
spaces. A proper submodule therefore has zero intersection
with the one dimensional highest weight space
$\mathbb Fv_{\boldsymbol\psi}$.
Hence the sum of all proper submodules is still proper
and is the unique maximal proper submodule.
We denote the resulting simple quotient by
$V_{\mathbb F}(\boldsymbol\psi)$.

We now recall Drinfeld's classification of finite dimensional
simple modules of $Y(\mathfrak g)$ over $\mathbb C$ \cite{D},  see also \cite[Theorem~3.1]{CPD}\cite[Chapter~3]{Mol07}\cite[Theorem~2.9]{TG}.

\begin{theorem*}[Drinfeld]
Every finite dimensional simple $Y(\mathfrak g)$-module
is highest weight. The simple highest weight module
$V_{\mathbb C}(\boldsymbol\psi)$ is finite dimensional
if and only if there are monic polynomials
$D_i\in\mathbb C[u]$ such that
\begin{equation}\label{eq:rep-charzero-highest}
\psi_i(u)=\frac{D_i(u+d_i)}{D_i(u)}
\qquad(i\in I),
\end{equation}
where the quotients are expanded at infinity.
These polynomials are unique.
\end{theorem*}

The polynomials $D_i$ are called the Drinfeld polynomials.
For the reduction modulo $p$ considered in Subsection \ref{subsection:stablelattice}, we will need
finite dimensional highest weight modules over fields of
characteristic zero other than $\mathbb C$.

\begin{remark*}
Theorem~\ref{thm:rank-basic} gives a direct proof of
the necessity of the polynomial condition \eqref{eq:rep-charzero-highest} in rank one.
Let $V$ be a finite dimensional highest weight
$Y(\mathfrak{sl}_2)$-module, generated by $v$, with
$h_0v=mv$.
Finite dimensional $\mathfrak{sl}_2$ representation
theory gives $m\in\mathbb N$.
The PBW theorem implies all weights belong to $m-2\mathbb N$,
and by the  complete reducibility of $\mathfrak{sl}_2$, we further obtain they are in the interval $[-m,m]$.
Hence $f_1^{(n)}v=0$ for $n>m$.

Apply \eqref{eq:rank-basic} with $a=b=n>m$ to $v$.
All summands with $t<n$ vanish because their positive
factors annihilate $v$, so
\[
 0=
 [u^{-n}]
 \left(
 \frac{Q(u)}{B_{m-n+1}(u)}
 \right)v.
\]
By \eqref{eq:rank-B-shifts-1} and $B_0(u)=1$,
\[
 B_{m-n+1}(u)^{-1}
 =
 \prod_{j=1}^{n-m-1}(1-j/u).
\]
Thus the preceding coefficient is $P_nv$ plus a
linear combination of $P_jv$ with $m<j<n$.
Induction on $n$, starting at $n=m+1$, gives
$P_nv=0$ for every $n>m$.

Consequently, there is a monic polynomial
$D(u)\in\mathbb C[u]$ of degree $m$ such that
\[
 Q(u)v=u^{-m}D(u)v.
\]
Equation~\eqref{eq:rank-difference} now gives
\[
 H(u)v=\frac{D(u+1)}{D(u)}v.
\]
This proves the necessary polynomial condition
directly from the Drinfeld presentation.
\end{remark*}

\subsection{Contravariant forms and extension of scalars}

To apply Drinfeld's theorem over an arbitrary field of
characteristic zero, we first show that simple highest weight
modules commute with extension of scalars.
We use a contravariant form whose radical is the maximal
proper submodule of the Verma module, which is the Yangian analogue of
the Shapovalov form for semisimple Lie algebras \cite{Sha}.

Let $\mathbb F$ be a field of characteristic zero, and let
\[
\epsilon^{<}:\mathcal Y_{\mathbb F}^{<}\longrightarrow\mathbb F,
\qquad
\epsilon^{>}:\mathcal Y_{\mathbb F}^{>}\longrightarrow\mathbb F
\]
be the augmentations sending all negative and positive
generators, respectively, to zero. The triangular decomposition
defines an $\mathbb F$-linear projection
\[
\HC:Y_{\mathbb F}(\mathfrak g)\longrightarrow
\mathcal H_{\mathbb F},
\qquad
\HC(x^-hx^+)
=
\epsilon^{<}(x^-)h\epsilon^{>}(x^+).
\]
This is the Harish--Chandra projection associated with the
triangular decomposition.

Fix a highest weight $\boldsymbol\psi$, and let
$\chi_{\boldsymbol\psi}:\mathcal H_{\mathbb F}\to\mathbb F$
be the character determined by
\[
\chi_{\boldsymbol\psi}(h_{i,r})=\mu_{i,r}
\qquad(i\in I,\ r\geq0).
\]
Let
$\ell\in\operatorname{Hom}_{\mathbb F}(M_{\mathbb F}(\boldsymbol\psi),\mathbb F)$
be the linear functional taking value $1$ on $v_{\boldsymbol\psi}$ and
vanishing on all other $\mathfrak g$-weight spaces. It follows that
\begin{equation}\label{eq:rep-highest-coefficient}
\ell(a v_{\boldsymbol\psi})=\chi_{\boldsymbol\psi}(\HC(a)),
\qquad a\in Y_{\mathbb F}(\mathfrak g).
\end{equation}

\begin{lemma}\label{lem:rep-contravariant-form}
The Verma module $M_{\mathbb F}(\boldsymbol\psi)$ admits a
unique bilinear form $B$ satisfying
\begin{equation}\label{eq:rep-contravariance}
B(v_{\boldsymbol\psi},v_{\boldsymbol\psi})=1,
\qquad
B(aw,z)=B(w,\tau(a)z)
\end{equation}
for all $a\in Y_{\mathbb F}(\mathfrak g)$ and
$w,z\in M_{\mathbb F}(\boldsymbol\psi)$.
The form is symmetric, and
\begin{equation}\label{eq:rep-HC-pairing}
B(xv_{\boldsymbol\psi},zv_{\boldsymbol\psi})
=
\chi_{\boldsymbol\psi}
\bigl(\HC(\tau(x)z)\bigr)
\qquad(x,z\in\mathcal Y_{\mathbb F}^{<}).
\end{equation}
Its radical is the unique maximal proper submodule of
$M_{\mathbb F}(\boldsymbol\psi)$. For every extension $\mathbb{E}/\mathbb F$ of characteristic zero
fields, there is an isomorphism
\begin{equation}\label{eq:rep-scalar-extension}
\mathbb{E}\otimes_{\mathbb F}V_{\mathbb F}(\boldsymbol\psi)
\simeq V_{\mathbb{E}}(\boldsymbol\psi).
\end{equation}
\end{lemma}

\begin{proof}
Write $M=M_{\mathbb F}(\boldsymbol\psi)$ and
$v=v_{\boldsymbol\psi}$. Define a left $Y_{\mathbb F}(\mathfrak g)$-module
structure on $\operatorname{Hom}_{\mathbb F}(M,\mathbb F)$
by
\[
 (a\varphi)(w)=\varphi(\tau(a)w).
\]
For every $w\in M$, we have
\[
(e_{i,r}\ell)(w)=\ell(f_{i,r}w)=0,
\qquad
(h_{i,r}\ell)(w)
=\ell(h_{i,r}w)=\mu_{i,r}\ell(w).
\]
The first equality follows because $f_{i,r}$ lowers weights
and $M$ has no weights above $\lambda_{\boldsymbol\psi}$.
For the second, $h_{i,r}$ preserves weight spaces and acts
on $\mathbb Fv$ by $\mu_{i,r}$.
Thus $\ell$ is a highest vector in $\operatorname{Hom}_{\mathbb F}(M,\mathbb F)$ with highest series $\boldsymbol\psi$. By \eqref{eq:rep-Verma-definition}, there is a
module homomorphism
\[
\Phi:M\longrightarrow
\operatorname{Hom}_{\mathbb F}(M,\mathbb F),
\qquad
\Phi(v)=\ell.
\]
Define
\[
B(w,z)=\Phi(w)(z).
\]
Then $B(v,v)=1$, and \eqref{eq:rep-contravariance} follows because $\Phi$ is a module homomorphism. For
$x,z\in\mathcal Y_{\mathbb F}^{<}$,
\eqref{eq:rep-highest-coefficient} yields
\[
B(xv,zv)=\ell(\tau(x)zv)=\chi_{\boldsymbol\psi}
  \bigl(\HC(\tau(x)z)\bigr),
\]
which proves \eqref{eq:rep-HC-pairing}.

Suppose that $B'$ also satisfies
\eqref{eq:rep-contravariance}. Since
$\tau(h_{i,0})=h_{i,0}$, distinct $\mathfrak g$-weight spaces are
orthogonal with respect to $B'$. Hence
\[
B'(v,M_\mu)=0
\qquad(\mu\ne\lambda_{\boldsymbol\psi}).
\]
Together with $B'(v,v)=1$, this gives $B'(v,-)=\ell$.
Consequently, for $x,z\in\mathcal Y_{\mathbb F}^{<}$,
\[
B'(xv,zv)
=\ell(\tau(x)zv)=B(xv,zv).
\]
Since $M=\mathcal Y_{\mathbb F}^{<}v$, the form is unique. The definition of $\tau$ and the triangular decomposition
give $\HC\circ\tau=\HC$. Therefore,
\[
B(xv,zv)=\chi_{\boldsymbol\psi}
  \bigl(\HC(\tau(x)z)\bigr)=\chi_{\boldsymbol\psi}
  \bigl(\HC(\tau(z)x)\bigr)=B(zv,xv),
\]
so $B$ is symmetric. Let
\[
\operatorname{rad}B
=\{w\in M\ |\ B(w,M)=0\}.
\]
By the second equation of \eqref{eq:rep-contravariance}, $\operatorname{rad}B$ is a submodule,
and $B(v,v)=1$ shows that it is proper. Let $N$ be any proper submodule of $M$. By the weight space
decomposition in the preceding subsection, $N$ has zero intersection with 
highest weight space. Hence $\ell(N)=0$.
For $w\in N$ and $x\in\mathcal Y_{\mathbb F}^{<}$,
we have $\tau(x)w\in N$, so
\[
B(xv,w)=\ell(\tau(x)w)=0.
\]
Since $M=\mathcal Y_{\mathbb F}^{<}v$ and $B$ is symmetric,
\[
N\subseteq\operatorname{rad}B.
\]
Thus $\operatorname{rad}B$ is the unique maximal proper
submodule, and $B$ induces a nondegenerate form on the simple quotient 
$V_{\mathbb F}(\boldsymbol\psi)$. Now let $\mathbb{E}/\mathbb F$ be a field extension. By the PBW theorem,
\[
M_\mathbb{E}(\boldsymbol\psi)
\simeq \mathbb{E}\otimes_{\mathbb F}M.
\]
Under this identification, the scalar extension $B_\mathbb{E}$
of $B$ satisfies \eqref{eq:rep-contravariance}, with $B_\mathbb{E}(v,v)=1$. We claim that
\[
\operatorname{rad}B_\mathbb{E}
=
\mathbb{E}\otimes_{\mathbb F}\operatorname{rad}B.
\]
Clearly, $\mathbb{E}\otimes_{\mathbb F}\operatorname{rad}B \subset \operatorname{rad}B_\mathbb{E}$. 
Conversely, let
\[
w=\sum_{j=1}^s a_j\otimes w_j
\in\operatorname{rad}B_\mathbb{E},
\]
where the $a_j\in \mathbb{E}$ are linearly independent over
$\mathbb F$. For every $z\in M$,
\[
0=B_\mathbb{E}(w,1\otimes z)
=\sum_{j=1}^s a_jB(w_j,z).
\]
Since $B(w_j,z)\in\mathbb F$, linear independence gives
\[
B(w_j,z)=0
\qquad(1\leq j\leq s,\ z\in M).
\]
Hence $w_j\in\operatorname{rad}B$ for every $j$,
proving the claim. It follows that
\[
\mathbb{E}\otimes_{\mathbb F}V_{\mathbb F}(\boldsymbol\psi)\simeq
\frac{\mathbb{E}\otimes_{\mathbb F}M}
     {\mathbb{E}\otimes_{\mathbb F}\operatorname{rad}B}\simeq
\frac{M_\mathbb{E}(\boldsymbol\psi)}
     {\operatorname{rad}B_\mathbb{E}}=V_\mathbb{E}(\boldsymbol\psi),
\]
which proves \eqref{eq:rep-scalar-extension}.
\end{proof}

\begin{corollary}\label{cor:rep-charzero-polynomials}
Let $\mathbb F$ be a field of characteristic zero, and let
$\mathbf D=(D_i)_{i\in I}$ be a tuple of monic polynomials
in $\mathbb F[u]$.
The simple highest weight module with highest series
\[
 \psi_i(u)=\frac{D_i(u+d_i)}{D_i(u)}
 \qquad(i\in I)
\]
is finite dimensional and absolutely simple.
We denote it by $V_{\mathbb F}(\mathbf D)$.
\end{corollary}

\begin{proof}
Let $\mathbb F_0\subseteq\mathbb F$ be the field generated
over $\mathbb Q$ by the coefficients of the $D_i$.
It is finitely generated over $\mathbb Q$, so it admits
an embedding $\sigma\colon \mathbb F_0\hookrightarrow\mathbb C$.
The highest series $\boldsymbol\psi$ have coefficients
in $\mathbb F_0$, and
\eqref{eq:rep-scalar-extension} gives
\[
 \mathbb C\otimes_{\sigma(\mathbb F_0)}
 V_{\mathbb F_0}(\boldsymbol\psi)
 \simeq
 V_{\mathbb C}(\sigma(\boldsymbol\psi)).
\]
The right hand side is finite dimensional by Drinfeld's
theorem, with Drinfeld polynomials $\sigma(D_i)$.
Since extension of scalars preserves the dimension of
a vector space, $V_{\mathbb F_0}(\boldsymbol\psi)$ is
finite dimensional. Applying
\eqref{eq:rep-scalar-extension} once more gives
\[
 V_{\mathbb F}(\boldsymbol\psi)
 \simeq
 \mathbb F\otimes_{\mathbb F_0}
 V_{\mathbb F_0}(\boldsymbol\psi),
\]
so this module is also finite dimensional and it remains simple after
extension of fields.
\end{proof}

Let $v$ be a highest vector of $V_{\mathbb F}(\mathbf D)$
and put $m_i=\deg D_i$.
Since
\[
 \frac{D_i(u+d_i)}{D_i(u)}
 =
 1+m_id_i u^{-1}+O(u^{-2}),
\]
we have $h_{i,0}v=m_iv$.
Thus the $\mathfrak g$-weight of $v$ is
$\lambda=\sum_{i\in I}m_i\omega_i$.
The weight decomposition of the Verma module gives
\[
 V_{\mathbb F}(\mathbf D)_\lambda=\mathbb Fv,
 \qquad
 \operatorname{wt}V_{\mathbb F}(\mathbf D)
 \subseteq\lambda-Q^+.
\]
In particular, every weight takes integer values on
the $h_{i,0}$.

\subsection{Lifting Drinfeld polynomials}

We now lift polynomial parameters over $\kk$ to characteristic
zero. Put
\[
 \mathcal O=W(\kk),
 \qquad
 \mathbb K=\operatorname{Frac}(\mathcal O),
\]
where $W(\kk)$ is the ring of  Witt vectors.
Since $\kk$ is perfect, $\mathcal O$ is a complete discrete
valuation ring of characteristic zero, with maximal ideal $p\mathcal O$, and
 $\mathcal O/p\mathcal O\simeq\kk$, 
see \cite[II.6]{SerreLF}.
In particular, every element of $\kk$ has a lift to
$\mathcal O$.
Such valuation rings are also used in the reduction of
representations of hyper loop algebras
\cite[Section~4]{JM}. Since $2$ is invertible in $\mathcal O$, the inclusion
$\mathbb Z\hookrightarrow\mathcal O$ extends to
$R\hookrightarrow\mathcal O$.
The notation introduced at the beginning of this section gives
\[
 \mathbb K\otimes_{\mathcal O}\mathcal Y_{\mathcal O}
 \simeq Y_{\mathbb K}(\mathfrak g),
 \qquad
 \kk\otimes_{\mathcal O}\mathcal Y_{\mathcal O}
 \simeq\mathscr Y_\kk(\mathfrak g).
\]
Let $\mathbf D=(D_i)_{i\in I}$ be a tuple of monic
polynomials in $\kk[u]$, with
\[
 D_i(u)=u^{m_i}+a_{i,1}u^{m_i-1}+\cdots+a_{i,m_i}.
\]
Choose lifts $\widetilde a_{i,j}\in\mathcal O$ of
their coefficients and set
\[
 \widetilde D_i(u)
 =
 u^{m_i}+\widetilde a_{i,1}u^{m_i-1}
 +\cdots+\widetilde a_{i,m_i}.
\]
Then $\widetilde D_i$ is monic of degree $m_i$ and
has image $D_i$ under the residue map
$\mathcal O[u]\to\kk[u]$. By Corollary~\ref{cor:rep-charzero-polynomials}, the tuple
$\widetilde{\mathbf D}=(\widetilde D_i)_{i\in I}$
determines a finite dimensional absolutely simple module
\[
 V=V_{\mathbb K}(\widetilde{\mathbf D}).
\]
Choose a highest vector $v$. Then
\begin{equation}\label{eq:rep-highest-series}
 e_{i,r}v=0,
 \qquad
 H_i(u)v=
 \frac{\widetilde D_i(u+d_i)}
      {\widetilde D_i(u)}v.
\end{equation}
Its $\mathfrak g$-weight is
$\lambda=\sum_{i\in I}m_i\omega_i$.
Let
\[
 \chi:\mathcal H_{\mathbb K}\longrightarrow\mathbb K
\]
be the character describing the Cartan action on
$\mathbb K v$.
To construct a lattice using the integral form, we first
compute $\chi$ on the generators of $\mathcal H_{\mathcal O}$.

\begin{lemma}\label{lem:rep-Cartan-character}
For $i\in I$ and $b\geq0$, one has
\begin{equation}\label{eq:rep-Cartan-character}
 \chi\left(\binom{h_{i,0}}b\right)=\binom{m_i}b,
 \qquad
 \chi(Q_i(u))=u^{-m_i}\widetilde D_i(u).
\end{equation}
Consequently,
$\chi(\mathcal H_{\mathcal O})\subseteq\mathcal O$.
\end{lemma}

\begin{proof}
Since $\chi(h_{i,0})=m_i$, the first identity follows
by evaluating the binomial polynomial at $m_i$. Applying $\chi$ to \eqref{eq:general-difference}, we obtain
\[
 (1+d_i/u)^{m_i}\chi(Q_i(u+d_i))
 =
 \frac{\widetilde D_i(u+d_i)}
      {\widetilde D_i(u)}\chi(Q_i(u)).
\]
On the other hand, the series
$q_i(u)=u^{-m_i}\widetilde D_i(u)$ also satisfies
\[
 (1+d_i/u)^{m_i}q_i(u+d_i)=
 u^{-m_i}\widetilde D_i(u+d_i)=
 \frac{\widetilde D_i(u+d_i)}
      {\widetilde D_i(u)}q_i(u)
\]
and has constant coefficient $1$. Since $\mathbb K$ has characteristic zero,
the recurrence \eqref{eq:general-P-recursion}
determines such a series uniquely. Hence
\[
 \chi(Q_i(u))=u^{-m_i}\widetilde D_i(u).
\]
Its coefficients belong to $\mathcal O$, so
$\chi(P_{i,j})\in\mathcal O$ for all $j\geq1$.
Together with $\binom{m_i}b\in\mathbb Z\subseteq\mathcal O$,
this gives $\chi(\mathcal H_{\mathcal O})\subseteq\mathcal O$.
\end{proof}

It follows that the coefficients of the lifted Drinfeld polynomials
are exactly the eigenvalues of the integral Cartan
generators $P_{i,m}$. In particular, $\mathcal H_{\mathcal O}v=\mathcal O v$.
We use this fact in the next subsection to construct
a $\mathcal Y_{\mathcal O}$-stable lattice in $V$.

\subsection{The cyclic lattice and its reduction}\label{subsection:stablelattice}

An $\mathcal O$-lattice in a finite dimensional
$\mathbb K$-vector space $V$ is a finite free
$\mathcal O$-submodule whose $\mathbb K$-span is $V$.
It is called stable if it is preserved by
$\mathcal Y_{\mathcal O}$.
For a stable lattice $\mathcal L$, the quotient
$\mathcal L/p\mathcal L$ is naturally a module for
the hyper Yangian $\mathscr Y_\kk(\mathfrak g)$.

\begin{proposition}\label{prop:rep-stable-lattice}
The module
\[
 \mathcal L(\widetilde{\mathbf D})
 =
 \mathcal Y_{\mathcal O}v
 =
 \mathcal Y_{\mathcal O}^{<}v
\]
is a stable $\mathcal O$-lattice in
$V=V_{\mathbb K}(\widetilde{\mathbf D})$.
Its quotient
\[
 \overline{\mathcal L}
 =
 \mathcal L(\widetilde{\mathbf D})/
 p\mathcal L(\widetilde{\mathbf D})
\]
is a finite dimensional
$\mathscr Y_\kk(\mathfrak g)$-module generated by the
nonzero highest vector $\overline v$, with
\begin{equation}\label{eq:rep-reduced-character}
 \binom{h_{i,0}}b\overline v
 =
 \binom{m_i}b\overline v,
 \qquad
 Q_i(u)\overline v
 =
 u^{-m_i}D_i(u)\overline v.
\end{equation}
Here $\binom{m_i}b$ is viewed as an element of $\kk$.
Moreover, $ \dim_\kk\overline{\mathcal L}
 =
 \dim_{\mathbb K}V$.
\end{proposition}

\begin{proof}
Put $\mathcal L=\mathcal Y_{\mathcal O}v$. The triangular decomposition and Lemma~\ref{lem:rep-Cartan-character} give
\[
 \mathcal L
 =
 \mathcal Y_{\mathcal O}^{<}
 \mathcal H_{\mathcal O}
 \mathcal Y_{\mathcal O}^{>}v
 =
 \mathcal Y_{\mathcal O}^{<}v.
\]
The module is stable by its definition, and its $\mathbb K$-span
is $V$, because $v$ generates $V$ over $Y_{\mathbb K}(\mathfrak g)$. It remains to prove that $\mathcal L$ is finite free
over $\mathcal O$.

Let $B$ be the nondegenerate bilinear form on $V$
induced by Lemma~\ref{lem:rep-contravariant-form}.  
Lemma~\ref{lem:rep-Cartan-character} in conjunction with \eqref{eq:rep-HC-pairing} implies 
\begin{equation}\label{eq:rep-integral-pairing}
 B(xv,zv)
 =
 \chi\bigl(\HC(\tau(x)z)\bigr)
 \in\mathcal O
 \qquad
 (x,z\in\mathcal Y_{\mathcal O}^{<}).
\end{equation}
In particular, $B(\mathcal L,\mathcal L)\subseteq\mathcal O$. 
Choose a $\mathbb K$-basis $w_1,\ldots,w_d$ of $V$ contained in
$\mathcal L$, and put
\[
 \mathcal L_0=\bigoplus_{j=1}^d\mathcal O w_j,
 \qquad
 \mathcal L_0^\#=
 \{w\in V\ |\ B(w,w_j)\in\mathcal O\text{ for }1\leq j\leq d\}.
\]
The linear map
\[
 V\longrightarrow\mathbb K^d,
 \qquad w\longmapsto(B(w,w_1),\ldots,B(w,w_d))
\]
is an isomorphism by  the nondegeneracy of $B$. Let $w_1^\#,\ldots,w_d^\#$ be the basis of $V$ determined by $B(w_j^\#,w_k)=\delta_{jk}$. Then 
\[\mathcal L_0^\#
 =\bigoplus_{j=1}^d\mathcal O w_j^\#
\]
is an $\mathcal O$-lattice of rank $d$ in $V$. Equation~\eqref{eq:rep-integral-pairing} gives
\[
 \mathcal L_0\subseteq\mathcal L\subseteq\mathcal L_0^\#.
\]
Since $\mathcal O$ is a discrete valuation ring, it is
a principal ideal domain.
Hence $\mathcal L$, as a submodule of the free module
$\mathcal L_0^\#$ of rank $d$, is free of rank at most $d$.
It contains $\mathcal L_0$, so its rank is exactly $d$.
This proves that $\mathcal L$ is a lattice.

We next show that the image of $v$ in
$\overline{\mathcal L}$ is nonzero.
If $v=pw\in p \mathcal L$ for some $w\in\mathcal L$, then
\[
 1=B(v,v)=pB(w,v)\in p\mathcal O,
\]
a contradiction. It follows that  $\overline{\mathcal L}$ is a non zero $\mathscr Y_\kk(\mathfrak g)$-module, remembering $\mathcal Y_{\mathcal O}/p\mathcal Y_{\mathcal O}
 \simeq\mathscr Y_\kk(\mathfrak g)$. The equality $\mathcal L=\mathcal Y_{\mathcal O}v$
shows that $\overline v$ generates this module, and \eqref{eq:rep-reduced-character} follows from \eqref{eq:rep-Cartan-character}.  Since $\mathcal L$ is free of rank $d$, we obtain  $\dim_\kk\overline{\mathcal L}
 =d=\dim_{\mathbb K}V$. 
\end{proof}

As a corollary, we readily obtain: 

\begin{corollary}\label{cor:rep-finite-realization}
For every tuple $\mathbf D=(D_i)_{i\in I}$ of monic
polynomials in $\kk[u]$, there is a finite dimensional
$\mathscr Y_\kk(\mathfrak g)$-module generated by a
nonzero highest vector $\overline v$ satisfying
\eqref{eq:rep-reduced-character} and
\begin{equation}\label{eq:rep-reduced-series}
 H_i(u)\overline v
 =
 \frac{D_i(u+d_i)}{D_i(u)}\overline v.
\end{equation}
\end{corollary}

\section{Finite dimensional simple modules of the small Yangian}
\label{sec:restricted-representations}

In this section, we classify the finite dimensional simple
modules of the small Yangian $y_\kk(\mathfrak g)$.
Theorem~\ref{thm:small-restricted} then gives the corresponding
classification for the restricted Yangian
$Y_\kk^{[p]}(\mathfrak g)$. To construct the required simple modules, we restrict
the modules from Corollary~\ref{cor:rep-finite-realization}
to $y_\kk(\mathfrak g)$ and take simple quotients of
the nonzero finite dimensional submodules generated
by their highest vectors. For the restricted Yangian of $\mathfrak{gl}_2$, such a
classification was obtained in \cite[Section~3]{CHT}. Throughout this section, we use the same notation
$e_{i,r}$, $f_{i,r}$, $h_{i,r}$, $e_\beta(r)$, and
$f_\beta(r)$ for their images in $y_\kk(\mathfrak g)$.

\subsection{Local nilpotence of the augmentation ideals}

The algebra $y_\kk(\mathfrak g)$ has a grading by the
root lattice $Q$, with
\[
 \deg e_{i,r}=\alpha_i,
 \qquad
 \deg f_{i,r}=-\alpha_i,
 \qquad
 \deg h_{i,r}=0.
\]
For $\gamma=\sum_i n_i\alpha_i\in Q^+$, write
$\operatorname{ht}(\gamma)=\sum_i n_i$. We use the loop filtration inherited from
$\mathscr Y_\kk(\mathfrak g)$.
By the proof of Theorem~\ref{thm:small-restricted},
$F_dy_\kk^{>}(\mathfrak g)$ has basis
\[
 \prod_{\beta,r}^{\rightarrow}
 e_\beta(r)^{a_{\beta,r}},
 \qquad
 0\leq a_{\beta,r}<p,
 \qquad
 \sum_{\beta,r}r a_{\beta,r}\leq d.
\]
The analogous statement holds for
$y_\kk^{<}(\mathfrak g)$ with the negative root vectors
in reverse order.

The assignment $e_{i,r}\mapsto0$ defines an algebra
homomorphism
\[
 \epsilon^{>}:y_\kk^{>}(\mathfrak g)\longrightarrow\kk,
 \qquad \epsilon^{>}(1)=1.
\]
Similarly, we define  $\epsilon^{<}:y_\kk^{<}(\mathfrak g)\longrightarrow\kk$. Their kernels
\[
 J^{>}=\ker\epsilon^{>},
 \qquad
 J^{<}=\ker\epsilon^{<}
\]
are the augmentation ideals. They are spanned by the
ordered product basis  containing at least one positive
or negative root vector, respectively.

\begin{lemma}\label{lem:rep-local-nilpotence}
The ideals $J^{>}$ and $J^{<}$ are locally nilpotent, that is, 
for every finite subset $\{x_1,\ldots,x_s\}$ of either
ideal, there is an integer $N\geq1$ such that
\[
 x_{j_1}\cdots x_{j_N}=0
 \qquad
 (1\leq j_1,\ldots,j_N\leq s).
\]
\end{lemma}

\begin{proof}
It suffices to prove the assertion for $J^{>}$, since
$\tau$ interchanges $J^{>}$ and $J^{<}$.
Put
\[
 c=(p-1)|\Delta^+|,
 \qquad
 h_*=\max_{\beta\in\Delta^+}\operatorname{ht}(\beta).
\]
For a PBW monomial
\[
 z=\prod_{\beta,r}^{\rightarrow}
 e_\beta(r)^{a_{\beta,r}},
 \qquad 0\leq a_{\beta,r}<p,
\]
put
\[
 n(z)=\sum_{\beta,r}a_{\beta,r},
 \qquad
 \operatorname{ht}(z)
 =\sum_{\beta,r}\operatorname{ht}(\beta)a_{\beta,r}.
\]
Then $\operatorname{ht}(z)\leq h_*n(z)$.
Write $A_r=\sum_\beta a_{\beta,r}$, then
$0\leq A_r\leq c$ and $\sum_r A_r=n(z)$.
Writing $n(z)=cq+t$ with $0\leq t<c$, we obtain
\begin{equation}\label{eq:rep-degree-bound}
\begin{aligned}
 \deg_{\mathrm{loop}}z
 &=\sum_{r\geq0}rA_r
 =qn(z)+\sum_{r\geq0}(r-q)A_r\\
 &\geq qn(z)+c\sum_{r=0}^{q-1}(r-q)
 =\frac{cq(q-1)}2+qt\\
 &=\frac{n(z)^2}{2c}-\frac{n(z)}2
   +\frac{t(c-t)}{2c}
 \geq\frac{n(z)^2}{2c}-\frac{n(z)}2.
\end{aligned}
\end{equation}
Fix $x_1,\ldots,x_s\in J^{>}$, and choose an integer
$L\geq0$ such that
$x_j\in F_Ly_\kk^{>}(\mathfrak g)$ for all $j$.
A product $x_{j_1}\cdots x_{j_N}$ of these elements lies in
$F_{NL}y_\kk^{>}(\mathfrak g)$.
Since every root homogeneous component of each $x_j$
has positive height, any PBW monomial $z$ occurring
in the product satisfies
\[
 \deg_{\mathrm{loop}}z\leq NL,
 \qquad
 \operatorname{ht}(z)\geq N,
 \qquad
 n(z)\geq N/h_*.
\]
For sufficiently large $N$, \eqref{eq:rep-degree-bound}
therefore gives
\[
 NL\geq \deg_{\mathrm{loop}}z
 \geq \frac{N^2}{2ch_*^2}-\frac{N}{2h_*}>NL,
\]
a contradiction. Thus every such product is zero
once $N$ is sufficiently large.
\end{proof}

\subsection{Highest weight modules}

Following \cite[Section~3.1]{CHT}, for a tuple
\[
 \boldsymbol\psi=(\psi_i(u))_{i\in I},
 \qquad
 \psi_i(u)
 =
 1+d_i\sum_{r\geq0}\mu_{i,r}u^{-r-1}
 \in1+u^{-1}\kk[[u^{-1}]],
\]
 we say that $\boldsymbol\psi$ is \emph{restricted} if
\begin{equation}\label{eq:rep-restricted-weight}
 \prod_{a\in\mathbb F_p}\psi_i(u+a d_i)=1
 \qquad(i\in I).
\end{equation}

\begin{lemma}\label{lem:rep-Cartan-characters}
The Cartan subalgebra has the presentation
\begin{equation}\label{eq:rep-Cartan-presentation}
 y_\kk^0(\mathfrak g)
 \simeq
 \kk[h_{i,r}\mid i\in I,\ r\geq0]/
 (\zeta_{i,r}\mid i\in I,\ r\geq0).
\end{equation}
Its characters are exactly the restricted tuples,
with $h_{i,r}\mapsto\mu_{i,r}$.
\end{lemma}

\begin{proof}

The Cartan generators of $y_\kk^0(\mathfrak g)$ satisfy
$\zeta_{i,r}=0$, so sending each variable $h_{i,r}$
to the corresponding Cartan generator defines a
surjective homomorphism
\[
 \kk[h_{i,r}\mid i\in I,\ r\geq0]/
 (\zeta_{i,r}\mid i\in I,\ r\geq0)
 \longrightarrow y_\kk^0(\mathfrak g).
\]
By \eqref{eq:small-Cartan-leading}, the Cartan monomials
with exponents less than $p$ span the quotient. Their images form a basis of
$y_\kk^0(\mathfrak g)$ by
Theorem~\ref{thm:small-restricted}.
Hence this homomorphism is also injective, proving
\eqref{eq:rep-Cartan-presentation}.  For a restricted tuple $\boldsymbol\psi$, consider the
evaluation homomorphism
\[
 \operatorname{ev}_{\boldsymbol\psi}:
 \kk[h_{i,r}\mid i\in I,\ r\geq0]\longrightarrow\kk,
 \qquad h_{i,r}\longmapsto\mu_{i,r}.
\]
It sends $N_i(u)$ to the left hand side of
\eqref{eq:rep-restricted-weight}, which equals $1$.
Hence $\operatorname{ev}_{\boldsymbol\psi}(\zeta_{i,r})=0$
for all $i,r$, so this homomorphism factors through
the quotient in \eqref{eq:rep-Cartan-presentation}
and defines a character of $y_\kk^0(\mathfrak g)$. Conversely, let $\chi:y_\kk^0(\mathfrak g)\to\kk$
be a character and put $\mu_{i,r}=\chi(h_{i,r})$.
Applying $\chi$ to the identity
$N_i(u)=1$ from Lemma~\ref{lem:small-norm-vanishing}
gives \eqref{eq:rep-restricted-weight}. Thus $\boldsymbol{\psi}$ is restricted.
\end{proof}

The Cartan commutations used in the proof of
Theorem~\ref{thm:small-restricted} give
\begin{equation}\label{eq:rep-Cartan-augmentation}
 y_\kk^0(\mathfrak g)J^{<}
\subseteq J^{<}y_\kk^0(\mathfrak g),\qquad
 J^{>}y_\kk^0(\mathfrak g)
 \subseteq y_\kk^0(\mathfrak g)J^{>}.
\end{equation}
Together with the triangular decomposition, these
relations show that every character of
$y_\kk^0(\mathfrak g)$ extends to the subalgebra
$y_\kk^0(\mathfrak g)y_\kk^{>}(\mathfrak g)$
by letting $J^{>}$ act as zero.

A highest weight $y_\kk(\mathfrak g)$-module with
highest weight $\boldsymbol\psi$ is generated
by a nonzero vector $v$ satisfying
\[
 e_{i,r}v=0,
 \qquad
 h_{i,r}v=\mu_{i,r}v
 \qquad(i\in I,\ r\geq0).
\]
Its highest weight is necessarily restricted by
Lemma~\ref{lem:rep-Cartan-characters}.

For a restricted tuple $\boldsymbol\psi$, let
$\kk_{\boldsymbol\psi}$ be the one dimensional
$y_\kk^0(\mathfrak g)y_\kk^{>}(\mathfrak g)$-module
on which $h_{i,r}$ acts by $\mu_{i,r}$ and
$J^{>}$ acts as zero.
Define  the induced module
\begin{equation}\label{eq:rep-baby-Verma}
 Z_\kk(\boldsymbol\psi)
 =
 y_\kk(\mathfrak g)
 \otimes_{y_\kk^0(\mathfrak g)y_\kk^{>}(\mathfrak g)}
 \kk_{\boldsymbol\psi},
 \qquad
 v_{\boldsymbol\psi}=1\otimes1.
\end{equation}
Equivalently, it is the quotient of $y_\kk(\mathfrak g)$
by the left ideal
\[
 \sum_{i,r}y_\kk(\mathfrak g)e_{i,r}
 +
 \sum_{i,r}y_\kk(\mathfrak g)(h_{i,r}-\mu_{i,r}).
\]
Every highest weight module with highest weight
$\boldsymbol\psi$ is therefore a quotient of
$Z_\kk(\boldsymbol\psi)$. By Corollary~\ref{cor:small-triangular}, the map
\begin{equation}\label{eq:rep-baby-PBW}
 y_\kk^{<}(\mathfrak g)
 \longrightarrow Z_\kk(\boldsymbol\psi),
 \qquad
 x\longmapsto xv_{\boldsymbol\psi}
\end{equation}
is an isomorphism of vector spaces, and the vectors
\[
 \left(
 \prod_{\beta,r}^{\leftarrow}
 f_\beta(r)^{a_{\beta,r}}
 \right)v_{\boldsymbol\psi},
 \qquad
 0\leq a_{\beta,r}<p,
\]
with finite support form a basis of $Z_\kk(\boldsymbol\psi)$. Following \cite{CHT}, we call
$Z_\kk(\boldsymbol\psi)$ the  \emph{baby Verma module}. The following result  extends
\cite[Proposition~3.4]{CHT}.

\begin{proposition}\label{prop:rep-simple-head}
The baby Verma module $Z_\kk(\boldsymbol\psi)$ has a unique maximal
proper submodule. Denote its simple quotient by
$L_\kk(\boldsymbol\psi)$. Moreover,
\[
 L_\kk(\boldsymbol\psi)
 \simeq
 L_\kk(\boldsymbol\phi)
 \quad\Longleftrightarrow\quad
 \boldsymbol\psi=\boldsymbol\phi.
\]
\end{proposition}

\begin{proof}
By \eqref{eq:rep-baby-PBW}, every vector of
$Z_\kk(\boldsymbol\psi)$ has a unique expression
\[
 (c+x)v_{\boldsymbol\psi},
 \qquad c\in\kk,\quad x\in J^{<}.
\]
By Lemma~\ref{lem:rep-local-nilpotence}, $x^n=0$
for some $n\geq1$. If $c\ne0$, then
\[
 (c+x)^{-1}
 =
 c^{-1}\sum_{j=0}^{n-1}(-c^{-1}x)^j
 \in y_\kk^{<}(\mathfrak g).
\]
Thus any submodule containing $(c+x)v_{\boldsymbol\psi}$
with $c\ne0$ contains $v_{\boldsymbol\psi}$ and equals
$Z_\kk(\boldsymbol\psi)$.
Every proper submodule is therefore contained in
$J^{<}v_{\boldsymbol\psi}$, hence the sum of all proper submodules is the unique maximal proper submodule. Finally, the one dimensional quotient $Z_\kk(\boldsymbol\psi)/J^{<}v_{\boldsymbol\psi}$ can be viewed as a $y^0_{\kk}(\fg)$-module, so we get $L_\kk(\boldsymbol\psi)
 \xrightarrow{\sim}
 L_\kk(\boldsymbol\phi)$ if and only if $\boldsymbol\psi=\boldsymbol\phi$.
\end{proof}

\begin{proposition}\label{prop:rep-highest-vector}
Every nonzero finite dimensional
$y_\kk(\mathfrak g)$-module contains a highest vector.
Consequently, every finite dimensional simple
$y_\kk(\mathfrak g)$-module is isomorphic to some
$L_\kk(\boldsymbol\psi)$.
\end{proposition}

\begin{proof}
Let $V$ be a nonzero finite dimensional
$y_\kk(\mathfrak g)$-module, and let $\mathcal E$ be
the image of $J^{>}$ in $\operatorname{End}_\kk(V)$.
Choose $x_1,\ldots,x_s\in J^{>}$ whose images form
a basis of $\mathcal E$.
By Lemma~\ref{lem:rep-local-nilpotence}, there is
an integer $N\geq1$ such that every product of $N$
of the $x_j$ is zero.
Every element of $\mathcal E$ is a linear combination
of their images. Expanding a product of $N$ such
linear combinations therefore gives zero.
Hence $\mathcal E^N=0$.

Choose the least $n\geq1$ such that $\mathcal E^nV=0$,
where $\mathcal E^0V=V$.
Then $\mathcal E^{n-1}V$ is nonzero and is annihilated
by $\mathcal E$. Hence
\[
 V^{>}
 =
 \{v\in V\ |\ J^{>}v=0\}
\]
is nonzero. By \eqref{eq:rep-Cartan-augmentation}, $V^{>}$ is
preserved by $y_\kk^0(\mathfrak g)$.
This algebra is commutative, and $\kk$ is algebraically
closed, so it has a nonzero common eigenvector in
$V^{>}$. This is a highest vector, whose highest weight
is restricted by Lemma~\ref{lem:rep-Cartan-characters}.

If $V$ is simple, the highest vector generates $V$,
so $V$ is a simple quotient of
$Z_\kk(\boldsymbol\psi)$ for some restricted $\boldsymbol{\psi}$.
Proposition~\ref{prop:rep-simple-head} implies $V$ is isomorphic to
$L_\kk(\boldsymbol\psi)$.
\end{proof}

\subsection{Drinfeld polynomials}

We now determine which restricted highest weights give
finite dimensional simple modules.

\begin{lemma}\label{lem:rep-highest-rationality}
Let $V$ be a nonzero finite dimensional highest weight
$y_\kk(\mathfrak g)$-module with highest weight
$\boldsymbol\psi$.
Then $\psi_i(u)\in\kk(u)$ for every $i\in I$.
\end{lemma}

\begin{proof}
Fix a highest vector $v$ and $i\in I$.
Since $V$ is finite dimensional, there exist $c_r\in\kk$ such that
\[
 \sum_{r=0}^N c_r f_{i,r}v=0,
 \qquad c_N\ne0.
\]
Applying $e_{i,s}$ and using $e_{i,s}v=0$, we obtain
\[
 0
 =
 \sum_{r=0}^N c_r[e_{i,s},f_{i,r}]v
 =
 \left(\sum_{r=0}^N c_r\mu_{i,r+s}\right)v.
\]
As $v\ne0$, this gives
\[
 \sum_{r=0}^N c_r\mu_{i,r+s}=0
 \qquad(s\geq0).
\]
Put $q_i(u)=\sum_{r=0}^N c_ru^r$.
For every $s\geq0$,
\[
 [u^{-s-1}]\bigl(q_i(u)\psi_i(u)\bigr)
 =
 d_i\sum_{r=0}^N c_r\mu_{i,r+s}
 =
 0.
\]
Thus $q_i(u)\psi_i(u)$ is a polynomial, and $\psi_i(u)$
is rational.
\end{proof}

In characteristic zero, a monic polynomial $D(u)$
is uniquely determined by the ratio $D(u+d)/D(u)$
for any fixed $d\ne0$.
In characteristic $p$, this is no longer true when
$d\in\mathbb F_p^\times$: multiplying $D(u)$ by
$u^p-u-c$ leaves the ratio unchanged, 
see \cite[Section~3.2]{Kal} and \cite[Remark~3]{CHT}.
To obtain unique monic polynomial parameters, we
exclude these factors. Call a monic polynomial $D\in\kk[u]$ \emph{$p$-reduced}
if it is not divisible by $u^p-u-c$ for any $c\in\kk$. For every $\xi\in\kk$, one has
\[
 u^p-u-(\xi^p-\xi)
 =
 \prod_{a\in\mathbb F_p}(u-\xi-a).
\]
Since $\kk$ is algebraically closed, $D$ is $p$-reduced
if and only if, for every $\xi\in\kk$, there is some
$a\in\mathbb F_p$ such that $D(\xi+a)\ne0$.

\begin{lemma}\label{lem:rep-rational-norm}
Let $d\in\mathbb F_p^\times$, and let $\psi(u)\in\kk(u)$
have expansion in $1+u^{-1}\kk[[u^{-1}]]$ at infinity.
Then
\begin{equation}\label{eq:rep-scalar-norm}
 \prod_{a\in\mathbb F_p}\psi(u+a d)=1
\end{equation}
if and only if there is a unique $p$-reduced polynomial
$D\in\kk[u]$ such that
\begin{equation}\label{eq:rep-scalar-ratio}
 \psi(u)=\frac{D(u+d)}{D(u)}.
\end{equation}
\end{lemma}

\begin{proof}
A ratio of the form \eqref{eq:rep-scalar-ratio}
satisfies \eqref{eq:rep-scalar-norm}, since
\[
 \prod_{a\in\mathbb F_p}
 \frac{D(u+(a+1)d)}{D(u+a d)}
 =
 1.
\]
Conversely, let $\sigma$ be the automorphism of
$\kk(u)$ given by $\sigma(f(u))=f(u+d)$.
The extension $\kk(u)/\kk(u^p-u)$ is cyclic of
degree $p$, with Galois group generated by $\sigma$.
Condition \eqref{eq:rep-scalar-norm} says that
$\psi$ has norm $1$. By Hilbert's Theorem~90
\cite[Chapter~X, \S1]{SerreLF},
there is $f\in\kk(u)^\times$ such that
\[
 \psi(u)=\frac{f(u+d)}{f(u)}.
\]
Write $f=A/B$ with nonzero $A,B\in\kk[u]$, and set
\[
 C(u)=\prod_{j=0}^{p-1}B(u+jd).
\]
Since $C(u+d)=C(u)$, the product
$g=Cf$ is a nonzero polynomial satisfying
\[
 \psi(u)=\frac{g(u+d)}{g(u)}.
\]
Remove from $g$ all factors of the form $u^p-u-c$
and divide by the leading coefficient, 
the resulting $p$-reduced polynomial still
satisfies \eqref{eq:rep-scalar-ratio}. The uniqueness follows from \cite[Remark, page 6985]{Kal}.
\end{proof}

\begin{theorem}\label{thm:rep-restricted-classification}
Let $\boldsymbol\psi$ be a restricted tuple.
The simple $y_\kk(\mathfrak g)$-module
$L_\kk(\boldsymbol\psi)$ is
finite dimensional if and only if there is a unique tuple of $p$-reduced
polynomials $D_i\in\kk[u]$ such that
\begin{equation}\label{eq:rep-polynomial-criterion}
 \psi_i(u)=\frac{D_i(u+d_i)}{D_i(u)}
 \qquad(i\in I).
\end{equation}
Consequently, tuples of $p$-reduced polynomials
parametrize the isomorphism classes of
finite dimensional simple $y_\kk(\mathfrak g)$-modules.
The same parametrization holds for
$Y_\kk^{[p]}(\mathfrak g)$ through the isomorphism
in Theorem~\ref{thm:small-restricted}.
\end{theorem}

\begin{proof}
Suppose first that
$L_\kk(\boldsymbol\psi)$ is
finite dimensional.
Lemma~\ref{lem:rep-highest-rationality} shows that
every $\psi_i(u)$ is rational.
Since $\boldsymbol\psi$ is restricted  and
$d_i\in\mathbb F_p^\times$,
Lemma~\ref{lem:rep-rational-norm} gives a unique
monic $p$-reduced polynomial $D_i$ satisfying
\eqref{eq:rep-polynomial-criterion}.

Conversely, suppose that such $p$-reduced polynomials $D_i$ exist.
By Corollary~\ref{cor:rep-finite-realization}, there is
a finite dimensional $\mathscr Y_\kk(\mathfrak g)$-module
containing a nonzero highest vector $\overline v$ with
\[
 H_i(u)\overline v
 =
 \frac{D_i(u+d_i)}{D_i(u)}\overline v.
\]
Restrict the action to $y_\kk(\mathfrak g)$, and put
$W=y_\kk(\mathfrak g)\overline v$. Then $W$ is a nonzero finite dimensional
$y_\kk(\mathfrak g)$-module generated by the highest
vector $\overline v$, with highest weight
$\boldsymbol\psi$.
By Proposition~\ref{prop:rep-simple-head}, its unique
simple quotient is
$L_\kk(\boldsymbol\psi)$,
which is therefore finite dimensional.
\end{proof}

\begin{remark*}
In view of \eqref{eq:rep-charzero-highest}, in characteristic zero, finite dimensional simple
 $Y(\fg)$-modules are parametrized by tuples of monic
Drinfeld polynomials.
In characteristic $p$, 
we call the polynomials $D_i$ in
Theorem~\ref{thm:rep-restricted-classification}
the \emph{restricted Drinfeld polynomials}
of $L_\kk(\boldsymbol\psi)$.
\end{remark*}

\bigskip
\noindent
\textbf{Acknowledgment.}
We thank Prof.~Naihuan Jing and  Prof.~Huafeng Zhang for helpful suggestions and comments. This work is supported by the National Natural 
Science Foundation of China (12401028, 12671035) and
the Natural Science Foundation of Hubei Province (2025AFB716). In the preparation of this paper, we used ChatGPT for mathematical discussions and assistance
with the exposition. In particular, ChatGPT assisted
us in developing and checking parts of the finite difference calculations
\eqref{eq:rank-Ac-shifts-1}--\eqref{eq:rank-Ac-shifts-2} and the recurrence \eqref{eq:rank-Z-recurrence} used in the proof of the rank one basic identity (Theorem \ref{thm:rank-basic}). It also helped clarify and check the Harish--Chandra expression \eqref{eq:rep-HC-pairing} for the bilinear form, and the proof of finite freeness using
the integral pairing (Proposition \ref{prop:rep-stable-lattice}).  All statements and proofs were independently checked by the authors, who take full
responsibility for the content of the paper.


\begin{thebibliography}{cccccccc}



\bibitem[BFM]{BFM}
Y.~Billig, V.~Futorny, A.~Molev,
{\em Verma modules for Yangians},
Lett. Math. Phys. \textbf{78} (2006), no.~1, 1--16.

\bibitem[BMM]{BMM}
A.~Bianchi, T.~Macedo, A.~Moura,
{\em On Demazure and local Weyl modules for affine hyperalgebras},
Pacific J. Math. \textbf{274} (2015), no.~2, 257--303.

\bibitem[BT]{BT}
J.~Brundan, L.~Topley,
{\em The $p$-centre of Yangians and shifted Yangians},
Mosc. Math. J. \textbf{18} (2018), no.~4, 617--657.

\bibitem[CHH]{CHH}
H.~Chang, H.~Hu, Y.~Hu,
{\em Shuffle algebra realizations for modular Yangians},
arXiv:2603.23243 (2026).


\bibitem[CHT]{CHT}
H.~Chang, J.~Hu, L.~Topley,
{\em Modular representations of the Yangian $Y_2$},
J. Lond. Math. Soc. (2) \textbf{111} (2025), no.~1, e70056.



\bibitem[CP1]{CPD}
V.~Chari, A.~Pressley,
{\em Yangians, integrable quantum systems and Dorey's rule},
Comm. Math. Phys. \textbf{181} (1996), no.~2, 265--302.

\bibitem[CP2]{CP}
V.~Chari, A.~Pressley,
{\em Quantum affine algebras at roots of unity},
Represent. Theory \textbf{1} (1997), 280--328.

\bibitem[D]{D}
V.~G.~Drinfeld,
{\em A new realization of Yangians and quantized affine algebras},
Soviet Math. Dokl. \textbf{36} (1988), 212--216.

\bibitem[Gar]{Gar}
H.~Garland,
{\em The arithmetic theory of loop algebras},
J. Algebra \textbf{53} (1978), no.~2, 480--551.

\bibitem[GT]{GT}
S.~M.~Goodwin, L.~Topley,
{\em Restricted shifted Yangians and restricted finite $W$-algebras},
Trans. Amer. Math. Soc. Ser. B \textbf{8} (2021), 190--228.

\bibitem[GTL]{GTL1}
S.~Gautam, V.~Toledano Laredo,
{\em Yangians and quantum loop algebras},
Selecta Math. (N.S.) \textbf{19} (2013), no.~2, 271--336.


\bibitem[Jan]{JanG}
J.~C.~Jantzen,
{\em Representations of Algebraic Groups},
2nd ed.,
Math. Surveys Monogr. \textbf{107},
Amer. Math. Soc., Providence, RI, 2003.

\bibitem[JLM]{JLM}
N.~Jing, M.~Liu and A.~Molev,
{\em Isomorphism between the $R$-matrix and Drinfeld
presentations of Yangian in types $B$, $C$ and $D$},
Comm. Math. Phys. \textbf{361} (2018),
827--872.

\bibitem[JM]{JM}
D.~Jakeli\'c, A.~Adrega de Moura,
{\em Finite-dimensional representations of hyper loop algebras},
Pacific J. Math. \textbf{233} (2007), no.~2, 371--402.


\bibitem[Kal]{Kal}
D.~Kalinov,
{\em Finite-dimensional representations of Yangians in complex rank},
Int. Math. Res. Not. IMRN (2020), no.~20, 6967--6998.

\bibitem[Kos]{Kos}
B.~Kostant,
{\em Groups over $\mathbb Z$},
in {\em Algebraic Groups and Discontinuous Subgroups},
Proc. Sympos. Pure Math. \textbf{9},
Amer. Math. Soc., Providence, RI, 1966, 90--98.

\bibitem[Lev]{Lev}
S.~Z.~Levendorski\u{\i},
{\em On PBW bases for Yangians},
Lett. Math. Phys. \textbf{27} (1993), no.~1, 37--42.

\bibitem[Lus]{Lus}
G.~Lusztig,
Quantum groups at roots of $1$,
Geom. Dedicata \textbf{35} (1990),
89--113.

\bibitem[Mol]{Mol07}
A.~I.~Molev,
{\em Yangians and classical Lie algebras},
Math. Surveys Monogr. \textbf{143},
Amer. Math. Soc., Providence, RI, 2007.

\bibitem[Ser]{SerreLF}
J.~P.~Serre,
{\em Local Fields},
Grad. Texts in Math. \textbf{67},
Springer-Verlag, New York, 1979.

\bibitem[Sha]{Sha}
N.~N.~Shapovalov,
{\em On a bilinear form on the universal enveloping algebra
of a complex semisimple Lie algebra},
Funct. Anal. Appl. \textbf{6} (1972), no.~4, 307--312.

\bibitem[TG]{TG}
Y.~Tan, N.~Guay,
{\em Local Weyl modules and cyclicity of tensor products for Yangians},
J. Algebra \textbf{432} (2015), 228--251.

\bibitem[Zha]{Zhang}
H.~Zhang,
{\em Theta series for quantum loop algebras and Yangians},
Commun. Math. Phys. \textbf{405} (2024), no.~10,
Paper No.~230.
\end{thebibliography}
\end{document}